\documentclass[12pt, reqno]{amsart}

\usepackage[utf8]{inputenc}
\usepackage{tikz}
\usepackage{mathbbol} 
\usepackage{amssymb,amsmath,amsfonts}
\usepackage{stmaryrd} 
	\DeclareSymbolFontAlphabet{\mathbb}{AMSb} 
	\DeclareSymbolFontAlphabet{\mathbbl}{bbold} 
\usepackage{tikz-cd}
\usetikzlibrary{arrows}
\usetikzlibrary{decorations.pathreplacing}
\usetikzlibrary{decorations.pathreplacing,decorations.markings,snakes}
\usepackage{color}
\usepackage{mathtools}
\usepackage{bussproofs}
\usepackage{mathrsfs} 
\usepackage{verbatim} 
\usepackage{hyphenat} 
\usepackage{hyperref} 
\hypersetup{
    colorlinks,
    linkcolor={red!50!black},
    citecolor={blue!50!black},
    urlcolor={blue!80!black}
}
\newcommand{\ind}{\mathbbl{1}} 
\newcommand{\alg}[1]{#1}
\newcommand{\vty}[1]{\mathbb{#1}}

\newcommand{\eq}{\approx}

\newtheorem{theorem}{Theorem}[section]
\newtheorem{proposition}[theorem]{Proposition}
\newtheorem{lemma}[theorem]{Lemma}
\newtheorem{corollary}[theorem]{Corollary}

\theoremstyle{definition}
\newtheorem{definition}[theorem]{Definition}
\newtheorem{remark}[theorem]{Remark}
\newtheorem{example}[theorem]{Example}

\newtheorem{problem}{Problem}

\renewcommand{\epsilon}{\varepsilon}
\renewcommand{\phi}{\varphi}
\newcommand{\lbr}{\llbracket}
\newcommand{\rbr}{\rrbracket}

\DeclareMathOperator{\Rg}{Reg}
\DeclareMathOperator{\Dn}{Den}
\DeclareMathOperator{\CoRg}{Reg^\partial}
\DeclareMathOperator{\CoDn}{Den^\partial}

\DeclareMathOperator{\C}{C} 

\DeclareSymbolFontAlphabet{\mathbb}{AMSb}
\DeclareSymbolFontAlphabet{\mathbbl}{bbold}

\begin{document}

\title{Hyper-MacNeille completions of Heyting algebras}
\author[J.~Harding]{John Harding}
\address{Department of Mathematical Sciences, New Mexico State University, Las Cruces, USA.}
\email{hardingj@mnsu.edu} 

\author[F.~M.~Lauridsen]{Frederik M\"{o}llerstr\"{o}m Lauridsen}
\address{Amsterdam, The Netherlands.}
\email{post@frederiklauridsen.eu}
\thanks{The second named author would like to thank Sam van Gool for helpful discussion and for drawing attention to~\cite{Spe69A, Spe69B}. This project has received funding from the European Union’s Horizon 2020 research and innovation program under the Marie Skłodowska-Curie grant agreement No 689176.}


\maketitle


\begin{abstract}
A Heyting algebra is supplemented if each element $a$ has a dual pseudo\hyp{}complement $a^+$, and a Heyting algebra is centrally supplement if it is supplemented and each supplement is central. We show that each Heyting algebra has a centrally supplemented extension in the same variety of Heyting algebras as the original. We use this tool to investigate a new type of completion of Heyting algebras arising in the context of algebraic proof theory, the so\hyp{}called hyper\hyp{}MacNeille completion. We show that the hyper\hyp{}MacNeille completion of a Heyting algebra is the MacNeille completion of its centrally supplemented extension. This provides an algebraic description of the hyper\hyp{}MacNeille completion of a Heyting algebra, allows development of further properties of the hyper\hyp{}MacNeille completion, and provides new examples of varieties of Heyting algebras that are closed under hyper\hyp{}MacNeille completions. In particular, connections between the centrally supplemented extension and Boolean products allow us to show that any finitely generated variety of Heyting algebras is closed under hyper\hyp{}MacNeille completions. 
\end{abstract}

\section{Introduction}
Recently, a unified approach to establishing the existence of cut\hyp{}free hypersequent calculi for various substructural logics has been developed~\cite{CGT17}. It is shown that there is a countably infinite set of equations/formulas, called $\mathcal{P}_3$, such that, in the presence of weakening and exchange, any logic axiomatized by formulas from $\mathcal{P}_3$ admits a cut\hyp{}free hypersequent calculus obtained by adding so\hyp{}called \emph{analytic structural rules} to a basic hypersequent calculus. 

The key idea is to establish completeness for the calculus without the cut\hyp{}rule with respect to a certain algebra and then to show that for any set of rules coming from $\mathcal{P}_3$\hyp{}formulas the calculus with the cut\hyp{}rule is also sound with respect to this algebra. In this way the argument is akin to the completeness\hyp{}via\hyp{}canonicity arguments known, most notably, from modal logic, see, e.g.,~\cite[Chap.~4]{BdRV01}, and has been employed in various levels of generality before to establish admissibility of the cut\hyp{}rule in different types of sequent calculi, see, e.g.,~\cite{GJ13, CGT12, GO10, W05, BJO04, Ono03, Oka02, Oka99, OT99, Mae91, OK85}. See also~\cite{Ter18} for a recent application of this method to second\hyp{}order logic and~\cite{GJLPT18} for similar methods applied in the context of display calculi. 

The algebras involved in the cut admissibility proof are based on complete lattices of closed sets determined by certain polarities called residuated (hyper)frames. Any algebra determines such a residuated (hyper)frame and hence gives rise to a complete algebra which turns out to be a completion of the original algebra. The analogous construction applied in the context of sequent calculi gives rise the well\hyp{}known MacNeille completion~\cite{CGT12, BJO04}. For this reason the completion introduced in~\cite{CGT17} is called the \emph{hyper\hyp{}MacNeille completion}. Restricting attention to intermediate logics, and hence to Heyting algebras, this provides a new method for completing Heyting algebras and shows that varieties of Heyting algebras defined by $\mathcal{P}_3$\hyp{}equations are closed under this hyper\hyp{}MacNeille completion. 

Our purpose here is to consider the hyper\hyp{}MacNeille completion of Heyting algebras from an algebraic, rather than proof\hyp{}theoretic, perspective. In doing so, we provide a simple alternative description of the hyper\hyp{}MacNeille completion of a Heyting algebra, are able to describe further the properties of this completion, and obtain results about varieties of Heyting algebras that are closed under hyper\hyp{}MacNeille completions but are not axiomatized by $\mathcal{P}_3$\hyp{}equations. In particular, we show that any finitely generated variety of Heyting algebras is closed under hyper\hyp{}MacNeille completions.    

Our primary tool is the notion of supplements. An element $a$ of a Heyting algebra is supplemented if it has a dual pseudo\hyp{}complement $a^+$. We say a Heyting algebra is centrally supplemented if it is supplemented and $a^+$ is central for each element $a$. We show that each Heyting algebra $\alg{A}$ has a centrally supplemented extension $S(\alg{A})$, and this extension belongs to the variety of Heyting algebras generated by the original. The technique is as follows. Any Heyting algebra $\alg{A}$ can be realized as a subdirect product $\alg{A}\leq\prod_Y\alg{A}_y$ of finitely subdirectly irreducible (fsi) quotients indexed over the minimum of its dual Esakia space. The full product is centrally supplemented, and we realize $S(\alg{A})$ as the supplemented subalgebra of the product generated by $\alg{A}$. 

We show that the hyper\hyp{}MacNeille completion $\alg{A}^+$ of a Heyting algebra $\alg{A}$ is the MacNeille completion of the centrally supplemented extension $S(\alg{A})$. This not only allows us to further develop the properties of the hyper\hyp{}MacNeille completion, but provides a simple algebraic description of hyper\hyp{}MacNeille completions for Heyting algebras. Finer properties of hyper\hyp{}MacNeille completions are obtained using a close connection between $S(\alg{A})$ and a sheaf representation of $\alg{A}$ over the subspace $Y$ consisting of the minimum of the dual Esakia space of $\alg{A}$. This leads to the interesting topic of describing the MacNeille completion of an algebra realized as the global sections of a sheaf. Some results in this vein are known~\cite{CHJ96} and employed here. 
 
The paper is organized as follows. Section~\ref{sec:supp} develops properties of supplemented bounded distributive lattices. Section~\ref{S(E)} develops the notion of a centrally supplemented extension of a Heyting algebra. Section~\ref{sec:HHMC} introduces the hyper\hyp{}MacNeille completion of a Heyting algebra and shows that the hyper\hyp{}MacNeille completion of any Heyting algebra is given by the MacNeille completion of its centrally supplemented extension. In Section~\ref{other} we relate our realization of the hyper\hyp{}MacNeille completion of a Heyting algebra to others. In Section~\ref{sheaf} we discuss relations to sheaf extensions, and in Section~\ref{examples} we provide examples of varieties that are closed under hyper\hyp{}MacNeille completions and are not defined by equations in $\mathcal{P}_3$. Finally, Section~\ref{sec:concluding-remarks} contains a few concluding remarks.       


\section{Supplemented Heyting algebras}\label{sec:supp}

A \emph{pseudo\hyp{}complement} of an element $a$ in a bounded distributive lattice $\alg{D}$, if it exists, is the largest element, denoted by $a^*$, whose meet with $a$ is 0. The notion of a supplement is dual to the notion of a pseudo\hyp{}complement. Thus, many basic properties of supplements can be obtained by dualizing known results~\cite[Chap.~VIII]{BD74} about pseudo\hyp{}complements. While we are primarily interested in the situation for Heyting algebras, the notion will be developed in the setting of bounded distributive lattices. Heyting algebras with supplements were studied, among others, by Sankappanavar in~\cite{San85}. We follow the notation used in that paper. 

\begin{definition}
For an element $a$ in a bounded distributive lattice $\alg{D}$, an element $a^+$ is the \emph{supplement} of $a$ if it is the least element whose join with $a$ is equal to $1$. We say $\alg{D}$ is \emph{supplemented} if each element of $\alg{D}$ has a supplement. 
\end{definition}  

It is easy to see, and well\hyp{}known dually from pseudo\hyp{}complements, that in a supplemented distributive lattice one always has $(x \land y)^+ \eq x^+ \lor y^+$. However, the identity $(x \lor y)^+ \eq x^+ \land y^+$ need not hold. This identity is dual to the usual Stone identity~\cite[Chap.~VIII.7]{BD74}. For an example, consider the lattice shown in Figure~\ref{fig}. 

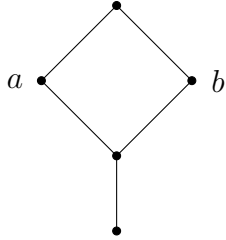
\begin{figure}[h]
\begin{center}
\begin{tikzpicture}[scale=1]
\draw[fill] (0,0) circle (1.5pt);
\draw[fill] (0,1) circle (1.5pt);
\draw[fill] (-1,2) circle (1.5pt);
\draw[fill] (1,2) circle (1.5pt);
\draw[fill] (0,3) circle (1.5pt);
\draw (0,0) -- (0,1) -- (-1,2) -- (0,3) -- (1,2) -- (0,1);
\node at (-1.35,2) {$a$};
\node at (1.35,2) {$b$};
\end{tikzpicture}
\end{center}
	\caption{A finite lattice not satisfying the dual Stone identity.}
	\label{fig}
\end{figure}

\begin{definition} \label{rb}
A bounded distributive lattice $\alg{D}$ is said to be \emph{centrally supplemented} if it is supplemented and satisfies the equation $(x\vee y)^+\eq x^+\wedge y^+$. 
\end{definition}

\begin{example}
Any finite distributive lattice is supplemented, but not necessarily centrally supplement. Examples of De Morgan supplemented distributive lattices include Boolean algebras, any bounded distributive lattice with a join\hyp{}irreducible top element, and any direct product of centrally supplemented distributive lattices. 
\end{example}

Recall, e.g., from~\cite[Chap.~VIII]{BD74}, that for a pseudo\hyp{}complemented distributive lattice $\alg{D}$, its set of \emph{regular elements} is $\Rg(\alg{D})= \{a^*:a\in\alg{D}\}$ and its set of \emph{dense elements} is $\Dn(\alg{D}) = \{a: a^*=0\}$.  Glivenko's Theorem~\cite[Thm.~VIII.4.3]{BD74} shows that $\Rg(\alg{D})$ is a Boolean algebra with the same bounds and meet as $\alg{D}$ and complementation given by pseudo\hyp{}complement. While $\Rg(\alg{D})$ is not a sublattice of $\alg{D}$, it is a quotient of $\alg{D}$ via the map sending $a$ to $a^{**}$ that collapses the filter of dense elements. We dualize these notions for a supplemented bounded distributive lattice $\alg{D}$ to obtain its set $\CoRg(\alg{D})=\{a^+:a\in\alg{D}\}$ of \emph{co\hyp{}regular} elements, and its set $\CoDn(\alg{D})=\{a:a^+=1\}$ of \emph{co\hyp{}dense} elements. The dual of the Glivenko theorem yields the following. 

\begin{proposition} \label{lem:coreg}
Let $\alg{D}$ be a supplemented bounded distributive lattice. Then $\CoRg(\alg{D})$ is a Boolean algebra with the same bounds and join as $\alg{D}$ and complementation given by supplement. In general $\CoRg(\alg{D})$ is not a sublattice, but it is a lattice quotient via the map sending $a$ to $a^{++}$ that collapses the ideal of co\hyp{}dense elements.  
\end{proposition} 

Recall that the center $Z(\alg{D})$ of a bounded distributive lattice $\alg{D}$ is the set of all of its complemented elements. The center is a bounded sublattice that forms a Boolean algebra. The following is a simple consequence of~\cite[Thm.VIII.7.1]{BD74}. 

\begin{proposition}\label{prop:deMorgan-eq-co-Stone}
For $\alg{D}$ a supplemented distributive lattice, the following are equivalent:
\begin{enumerate}
\item $Z(\alg{D}) = \CoRg(\alg{D})$,
\item The algebra $\alg{D}$ satisfies the equation $x^+ \land x^{++} \eq 0$, 
\item $\CoRg(\alg{D})$ is a sublattice of $\alg{D}$, thus a retract of $\alg{D}$,
\item The supplement on $\alg{D}$ is a central supplement.
\end{enumerate}
\end{proposition}

\begin{remark}
Thus Proposition~\ref{prop:deMorgan-eq-co-Stone} shows that the bounded distributive lattices that are centrally supplemented are precisely the bounded distributive lattice with order duals being \emph{Stone lattices}, see~\cite{GS57} or~\cite[Chap.~VIII.7]{BD74}. It also points to the reason for the term ``centrally supplemented''. One further advantage of this term is it allows us to speak of the property of being centrally supplemented for an individual element $a$. We say $a$ is \emph{centrally supplemented} if $a$ has a supplement $a^+$ and this supplement belongs to the center. 
\end{remark}

If a bounded distributive lattice is complete, pseudo\hyp{}complemented, and supplemented, then its center is well\hyp{}behaved. 

\begin{proposition}\label{prop:completecenter}
Let $\alg{D}$ be a bounded distributive lattice which is both pseudo\hyp{}complemented and supplemented. If $\alg{D}$ is complete, then so is $Z(\alg{D})$ and joins and meets in $Z(\alg{D})$ are the same as those in $\alg{D}$. 
\end{proposition} 

\begin{proof}
Let $\{c_i\}_{i \in I}$ be a collection of central elements. Note that for central elements, their pseudo\hyp{}complements and supplements agree, and we write these as complements $c_i'$. Let $c = \bigwedge_{I} c_i$ be the meet in $\alg{D}$ and let $d = \bigvee_{I} c'_i$ be the join in $\alg{D}$. For each $i\in I$ we have $c\leq c_i$, so $c_i'=c_i^*\leq c^*$, hence $d\leq c^*$. Thus $c\wedge d = 0$. Similarly, since $c=\bigwedge_Ic_i''$ we have $c\vee d=1$. So $c$ has a complement, hence belongs to the center, and is therefore the meet of the family $\{c_i\}_{i\in I}$ in the center as well. The statement for joins follows by a dual argument.  
\end{proof} 

\begin{lemma}\label{lem:meet-dense-supplemented-set}
Let $\alg{D}$ be a complete bounded distributive lattice and let $S \subseteq D$ be a meet-dense set. If every element of $S$ has a supplement in $\alg{D}$, then $\alg{D}$ is supplemented and the supplement of an element $x\in\alg{D}$ is given by 
\[ x^+ = \bigvee \{s^+ \in D : x \leq s \in S\}.\]
\end{lemma}

\begin{proof}
Let $x \in D$ and set $z = \bigvee \{s^+ \in D : x \leq s \in S\}$. Note that if $y \in D$ is such that $y \lor x = 1$ then for each $s \in S$  with $s \geq x$ we have that $y \lor s = 1$ and hence $s^+ \leq y$, showing that $ z \leq y$. To see that  $x \lor z=1$, consider $s \in S$ such that $x \lor z \leq s$. Then $x \leq s$ and $z \leq s$ from which we may conclude that $s^+ \leq s$ and therefore that $s = 1$. Since $S$ is assumed to be meet-dense in $\alg{D}$ we may conclude that $x \lor z =1$. 
\end{proof}

We next turn to the matter of MacNeille completions. For a distributive lattice $\alg{D}$ we let $\overline{\alg{D}}$ be its MacNeille completion. It is well\hyp{}known that there are bounded distributive lattices whose MacNeille completions are not distributive, see, e.g.,~\cite[Chap.~XII.2]{BD74}, and by placing a new top on such, we obtain a (centrally) supplemented distributive lattice whose MacNeille completion is not distributive. But our interest here is in completions of Heyting algebras, and in this setting things are better behaved. 

\begin{proposition}\label{prop:MN-of-DM-is-DM}
If $\alg{A}$ is a centrally supplemented Heyting algebra, then its MacNeille completion $\overline{\alg{A}}$ is again a centrally supplemented Heyting algebra and $\alg{A}$ is a supplemented Heyting subalgebra of $\overline{\alg{A}}$. 
\end{proposition}

\begin{proof}
Let $\alg{A}$ be a centrally supplemented Heyting algebra. It is well known that the MacNeille completion of a Heyting algebra is also a Heyting algebra, see, e.g.,~\cite[Thm.~2.3]{HB04}. If $a\in\alg{A}$ and $b$ is the supplement of $a$ in $\alg{A}$, then the meet\hyp{}density of $\alg{A}$ in $\overline{\alg{A}}$ entails that $b$ is also the supplement of $a$ in $\overline{\alg{A}}$. From Lemma~\ref{lem:meet-dense-supplemented-set} it then follows that $\overline{\alg{A}}$ is a supplemented Heyting algebra. To see that this supplement is central, suppose $x\in\overline{A}$. By Lemma~\ref{lem:meet-dense-supplemented-set} we have $x^+ = \bigvee\{a^+ : x \leq a \in A\}$. Since $\alg{A}$ is a Heyting algebra it follows that 
\[
x^+ \land x^{++} = \bigvee\{a^+ \land b^+ : x \leq a \in A \; \text{and} \; x^+ \leq b \in A\}.
\]   
If $a, b \in A$ are such that $x \leq a$ and $x^+ \leq b$, then $b^+ \leq x^{++} \leq x \leq a$, so $a^+ \leq b^{++}$. It then follows that $a^+ \land b^+ \leq b^{++} \land b^+$. But as $\alg{A}$ was assumed to be centrally supplemented, $b^{++} \land b^+ =0$. So $x^+ \land x^{++} = 0$, as desired. 
\end{proof}

\begin{proposition} \label{center}
If $A$ is a centrally supplemented Heyting algebra, then $Z(\overline{A})=\overline{Z(A)}$ and is a complete subalgebra of $\overline{A}$. 
\end{proposition}

\begin{proof}
By Proposition~\ref{prop:MN-of-DM-is-DM}, $\overline{A}$ is centrally supplemented so by Proposition~\ref{prop:completecenter} we have $Z(\overline{A})$ is complete and a complete subalgebra of $\overline{A}$. Since $A\leq\overline{A}$ we have $Z(A)\leq Z(\overline{A})$. Since both are Boolean, to show that $Z(\overline{A})$ is the MacNeille completion of $Z(A)$ it suffices to show that $Z(A)$ is join\hyp{}dense in $Z(\overline{A})$. For $x\in Z(\overline{A})$ we have $x=x^{++}$. So by Lemma~\ref{lem:meet-dense-supplemented-set} $x=\bigvee\{a^+:a\in A\mbox{ and }a\leq x^+\}$. Here this join is in $\overline{A}$, but each element $a^+$ belongs to $Z(A)$ since $A$ is centrally supplemented, hence this is a join of elements in $Z(\overline{A})$, and therefore is the join taken in $Z(\overline{A})$ since this is a complete subalgebra of $\overline{A}$. 
\end{proof}

We assume familiarity with Priestley (and Esakia) duality, see, e.g.,~\cite{DP02,Esa85,Esa19}. For a bounded distributive lattice $\alg{D}$ we will let $X_\alg{D}$ denote its Priestley space. This consists of the prime filters on $\alg{D}$, partially ordered by inclusion, equipped with the topology generated by (cl)opens of the form $\widehat{a}$ and their complements, with $\widehat{a} \coloneqq \{x \in X_\alg{D} : a \in x\}$ for $a \in D$. Any bounded distributive lattice $\alg{D}$ can be recovered from $X_\alg{D}$ as the lattice of clopen upsets. We will make much use of the fact that any element in a Priestley space is above some minimal element see, e.g.,~\cite{Esa85,Esa19} or~\cite[Thm.~2.3.24]{Bez06}, and will let $\min(X_\alg{D})$ be the set of minimal elements of $X_\alg{D}$ equipped with the subspace topology. 

\begin{lemma}\label{min}
Let $\alg{D}$ be a bounded distributive lattice. 
\begin{enumerate}
\item If $x$ is a minimal prime filter and $a\in D$ is such that $x\in\widehat{a}$, then $a\vee s = 1$ for some $s\in D$ with $s\not\in x$.
\item The subspace topology on $\min(X_\alg{D})$ has as a basis $\{\min(X_\alg{D})\setminus\widehat{s}:s\in\alg{D}\}$. 
\item If $c$ is co\hyp{}dense then $\widehat{c}$ is disjoint from $\min(X_\alg{D})$. 
\item If $\alg{D}$ is supplemented, then for each $a\in D$ and each minimal prime filter $x$, exactly one of $a,a^+$ belongs to $x$. 
\item If $\alg{D}$ is centrally supplemented, then its minimal prime filters are exactly the filters generated by ultrafilters of the center.  
\end{enumerate}
\end{lemma}

\begin{proof}
Let $x$ is a minimal prime filter with $x\in\widehat{a}$ for some $a \in D$. For such $x$ and $a$, basic properties of Priestley spaces give ${\uparrow}(X_\alg{D}\setminus\widehat{a})$ is a closed upset that is disjoint from the closed downset $\{x\}$. So by~\cite[Lem.~11.21(ii)(b)]{DP02} there is a clopen upset, say $\widehat{s}$, that contains ${\uparrow}(X_\alg{D}\setminus\widehat{a})$ and is disjoint from $\{x\}$. But then $\widehat{a\vee s}$ must contain both $\widehat{a}$ and $X_\alg{D}\setminus\widehat{a}$ so $a\vee s = 1$, and as $x\not\in\widehat{s}$ we have $s\not\in x$. The statements~(2) and~(3) are direct consequences of (1). For the second to last statement note that $a\vee a^+=1$ so at least one of $a,a^+$ belongs to any prime filter $x$. But $a\wedge a^+$ is co\hyp{}dense, so if $x$ is minimal then by~(3) at most one of $a,a^+$ belongs to $x$. For the final statement, if $x$ is a minimal prime filter, then it follows from~(4) that if $a\in x$, then the central element $a^{++}\leq a$ belongs to $x$. So $x$ is generated by its central elements which must be an ultrafilter of the center. Conversely, let $u$ be an ultrafilter of the center and consider the filter ${\uparrow}u$ of $D$ generated by $u$. Note that $a\in{\uparrow}u$ iff $a^{++}\in u$. So $a\vee b\in{\uparrow}u$ implies $a^{++}\vee b^{++}\in u$, and as $u$ is prime, that ${\uparrow}u$ is prime. Since every prime filter contains a unique ultrafilter of the center, it follows that ${\uparrow}u$ is a minimal prime filter. 
\end{proof}

\begin{proposition}[{cf.~\cite[Thm.~1]{Spe69B}}]\label{prop:nable-ast-iff-closed-min}
For $\alg{D}$ a bounded distributive lattice, the following are equivalent:
\begin{enumerate}
\item  $\min(X_\alg{D})$ is a closed subset of $X_\alg{D}$, 
\item For each $a\in\alg{D}$ there is $b \in D$ with $a\vee b=1$ and $a\wedge b \in \CoDn(D)$. 
\end{enumerate}
\end{proposition}

Bounded distributive lattices satisfying the order dual of condition~(2) above were called $\Delta_\ast$\hyp{}lattices by Speed~\cite{Spe69B}. 

\begin{proposition}\label{prop:Dual-BA-of-min}
Let $\alg{D}$ be a supplemented bounded distributive lattice. Then $\min(X_\alg{D})$ is a Stone space, the dual Boolean algebra of which is $\CoRg(\alg{D})$. If $\alg{D}$ is centrally supplemented, then this dual Boolean algebra is the center $Z(\alg{D})$. 
\end{proposition}

\begin{proof}
Note first that any supplemented distributive lattice is a dual $\Delta_\ast$\hyp{}lattice since $a\wedge a^+$ is co\hyp{}dense. It follows from Proposition~\ref{prop:nable-ast-iff-closed-min} that $\min(X_\alg{D})$ is a Stone space since it is a closed subspace of a Priestly space. 
By~\cite[Lem.~12]{Pri72} we have that the dual algebra of the Stone space $\min(X_\alg{D})$ is the homomorphic image of $\alg{D}$ determined by the congruence $\theta$ given by
\[
a \theta b \iff \forall x \in \min(X_\alg{D}) \ (a \in x \iff b \in x). 
\]
Using Lemma~\ref{min} it is not hard to show that $a \theta b$ iff $a^+ = b^+$. It follows that $a \theta b$ iff $a^{++}=b^{++}$, hence $\theta$ is the kernel of the homomorphism from $\alg{D}$ onto $\CoRg(\alg{D})$ taking $a$ to $a^{++}$, and our result follows. 

The final comment follows from Proposition~\ref{prop:deMorgan-eq-co-Stone}.
\end{proof}

\begin{remark}
Centrally supplemented Heyting algebras have a further interesting property. Note that $((x \leftrightarrow y)^+ \land x) \lor ((x\leftrightarrow y)^{+*} \land z)$ gives a discriminator term on any fsi Heyting algebra~\cite[Sec.~5]{San85}. Consequently, considering supplementation as part of the type, the class of centrally supplemented Heyting algebras form a discriminator variety, cf.,~\cite[Chap.~IV.9]{BS81}. 
\end{remark}


\section{The centrally supplemented extension} \label{S(E)}

We describe a particular embedding of a Heyting algebra $\alg{A}$ into a centrally supplemented Heyting algebra $S(\alg{A})$. Much of our work could be done in the setting of bounded distributive lattices, but we use it only for Heyting algebras, and develop it there. For convenience, we use $Y_\alg{A}$, or simply $Y$ when no confusion is likely, for the minimum $\min(X_\alg{A})$ of the dual Esakia space $X_{\alg{A}}$ of a Heyting algebra $\alg{A}$. Each element $y\in Y$ is a filter of $\alg{A}$ so gives a congruence $\theta_y$ where 
\[
a \mathrel{\theta_y} b \iff a \wedge c = b\wedge c \mbox{ for some }c\in y.
\]
We let $\alg{A}_y$ be the quotient $\alg{A}/\theta_y$ and $P(\alg{A})$ be the product $\prod_Y\alg{A}_y$. The following is the crucial observation for us. 

\begin{proposition}\label{q}
For a Heyting algebra $\alg{A}$ with $Y$ the minimum of its dual Esakia space, $\alg{A}_y$ is a centrally supplemented Heyting algebra for each $y\in Y$, the product $P(\alg{A})=\prod_Y\alg{A}_y$ is a centrally supplemented, and the natural map of $\alg{A}$ into $P(\alg{A})$ is a subdirect embedding that preserves existing central supplements. 
\end{proposition}

\begin{proof}
That $\alg{A}_y$ is centrally supplemented for each $y\in Y$ follows from the fact that $y$ is a prime filter, so 1 is join\hyp{}irreducible in $\alg{A}_y$. It follows that the product $P(\alg{A})$ is also centrally supplemented. To see that the natural map is a subdirect embedding, we need only show that it is one\hyp{}to\hyp{}one. Let $a\neq b$ in $\alg{A}$ and set $I=\{c \in A : a\wedge c = b\wedge c\}$. Then $I$ is a proper ideal, so there is a prime filter disjoint from it, hence a minimal prime filter $y$ disjoint from $I$. Then we have $a/\theta_y\neq b/\theta_y$. Let $a\in\alg{A}$ be centrally supplemented with supplement $a^+$. By Lemma~\ref{min} for a minimal prime filter $y$ we have $a^+\in y$ iff $a\not\in y$. So the function in the product corresponding to $a^+$ takes value 1 iff the function corresponding to $a$ takes value less than 1, and is 0 otherwise. This is the supplement in the product. 
\end{proof}

A word on notation. We treat $\alg{A}$ as if it is a subalgebra of the product $P(\alg{A})$ identifying an element $a\in \alg{A}$ with the function with $a(y)=a/\theta_y$ for each $y\in Y$. For two functions $f,g$ we use $\lbr f=g\rbr$ and so forth to denote the obvious subset of $Y$ and we let $\ind_{(f=g)}$ be the characteristic function of this subset. So 
\[ 
\ind_{(f=g)}(y) = \begin{cases} 1 & \mbox{if } f(y)=g(y),\\ 0 & \mbox{else.} \end{cases} 
\]
Since the factors $\alg{A}_y$ for $y\in Y_\alg{A}$ have their unit $1_y$ join\hyp{}irreducible, it is easily seen that for an element $f$ of the product $P(\alg{A})$, its supplement $f^+$ is given by $\ind_{(f<1)}$, which is clearly central. 

\begin{definition} \label{S(A)}
For a Heyting algebra $\alg{A}$ let $S(\alg{A})$ be the supplemented Heyting subalgebra of the product $P(\alg{A})=\prod_{Y}\alg{A}_y$ that is generated by $\alg{A}$. 
\end{definition}

\begin{proposition} \label{D}
For a Heyting algebra $\alg{A}$, the set $D(\alg{A})=\{\ind_{(a=1)}:a\in\alg{A}\}$ is a bounded distributive sublattice of the center of $S(\alg{A})$.
\end{proposition}

\begin{proof}
Since the factors $\alg{A}_y$ are fsi, we have that $(a\vee b)(y)=1$ iff $a(y)=1$ or $b(y)=1$ and clearly $(a\wedge b)(y)=1$ iff both $a(y)=1$ and $b(y)=1$. So $\alg{D}$ is a bounded distributive sublattice of $P(\alg{A})$. Since $\ind_{(a=1)}=(\ind_{(a<1)})^+=a^{++}$ both $\ind_{(a=1)}$ and $\ind_{(a<1)}$ belong to $S(\alg{A})$. Since they are complements of each other, both belong to the center.
\end{proof}

\begin{definition}
For a Heyting algebra $\alg{A}$, let $B(\alg{A})$ be the Boolean subalgebra of the center of $S(\alg{A})$ that is generated by $\alg{D}(\alg{A})$. 
\end{definition}

Recall that a \emph{finite partition of unity} in a Boolean algebra $B$ is a finite sequence $e_1,\ldots,e_n$ of pairwise disjoint non\hyp{}zero members of $B$ that join to 1. 

\begin{proposition}\label{fr}
Let $\alg{A}$ be a Heyting algebra. For any two elements $u,v\in S(\alg{A})$ there is a finite partition of unity $e_1,\ldots,e_n$ in $B(\alg{A})$ and elements $a_1,\ldots,a_n,b_1,\ldots,b_n\in\alg{A}$ with 
\[ 
u=  \bigvee_{i=1}^n a_i\wedge e_i \qquad\mbox{ and }\qquad v=  \bigvee_{i=1}^n b_i\wedge e_i \tag{$\star$}
\]
Further, $e_1,\ldots,e_n$ can be chosen so that $e_i=\ind_{(c_i=1)}\wedge\ind_{(d_i<1)}$ for some finite even length sequence $d_1<c_1\leq\ldots\leq d_n<c_n$ in $\alg{A}$ with $d_1=0$ and $c_n=1$. 
\end{proposition}

\begin{proof}
Let $T$ be the set of all elements of $P(\alg{A})$ that are equal to $\bigvee_{i=1}^n a_i\wedge e_i$ for some finite partition of unity $e_1,\ldots, e_n$ of $B(\alg{A})$ and some $a_1,\ldots,a_n \in\alg{A}$. Since $B(\alg{A})$ is contained in the center of  $S(\alg{A})$ and each $a\in\alg{A}$ is in $S(\alg{A})$, we have that $T$ is contained in $S(\alg{A})$. 

Take two elements $u,v\in T$. Then there are finite partitions of unity $f_1,\ldots,f_k$ and $g_1,\ldots,g_m$ and elements $r_1,\ldots,r_k$ and $s_1,\ldots, s_m$ with $u=\bigvee_{j=1}^k r_j\wedge f_j$ and $v=\bigvee_{\ell=1}^m s_\ell\wedge g_\ell$. These two finite partitions of unity have a common refinement $e_1,\ldots,e_n$, meaning that for each $i\leq n$ there are unique $j\leq k$ and $\ell\leq m$ with $e_i\leq f_j,g_\ell$. Setting $a_i=r_j$ and $b_i=s_\ell$ we have a representation of $u,v$ as in ($\star$) of the statement of the proposition. We then have the following description for the meet, join, and Heyting implication of $u$ and $v$ in the product $P(\alg{A})$. 
\[ 
u\wedge v= \bigvee_{i=1}^n (a_i\wedge b_i)\wedge e_i \qquad u\vee v= \bigvee_{i=1}^n (a_i\vee b_i)\wedge e_i \qquad u\to v= \bigvee_{i=1}^n (a_i\to b_i)\wedge e_i
\]
This is seen by noting that each $e_i$ is the characteristic function of some subset $E_i$ of the indexing set $Y$ and the sets $E_i$ are pairwise disjoint, non-empty, and cover $Y$. So the descriptions in ($\star$) say that $u$ takes the same value as $a_i$ on $E_i$ and $v$ the same value as $b_i$ on $E_i$. So the results of the operations are just the component\hyp{}wise ones. It follows from this that $T$ is a Heyting subalgebra of $P(\alg{A})$. 

Since the supplement of the product is central, supplements satisfy the De Morgan laws with respect to meet and join. So,  
\[u^+ = \bigwedge_{i=1}^n a_i^+\vee e_i^+\]
Since $a_i^+=\ind_{(a_i<1)}$ belongs to $B(\alg{A})$ and each $e_i^+$ is the complement of $e_i$ so belongs to $B(\alg{A})$ we have that $u^+$ belongs to $B(\alg{A})$, and hence to $T$. Thus $T$ is a supplemented Heyting subalgebra of $P(\alg{A})$ that contains $\alg{A}$, and therefore contains $S(\alg{A})$. Therefore $T=S(\alg{A})$. 

It remains to show the further remark. Any finite partition of unity in the Boolean algebra $B$ generated by a distributive lattice $D$ is a finite partition of unity in the Boolean algebra generated by some finite distributive sublattice $D'$ of $D$, and hence~\cite[Chap.~V.7]{BD74} is a finite partition of unity in the Boolean algebra $B(C)$ generated by some finite chain $C$ of $D$. The atoms of this Boolean algebra are of the form $p\wedge q'$ where $p$ covers $q$ in the chain. Thus any finite partition of unity $e_1,\ldots,e_n$ in $B(\alg{A})$ has a refinement to one given by the atoms of the Boolean algebra $B(C)$ for some finite chain $C$ in $D=\{\ind_{(a=1)}:a\in\alg{A}\}$, and our result follows. 
\end{proof} 
 
\begin{proposition} \label{k1}
For a Heyting algebra $\alg{A}$, the centrally supplemented Heyting extension $S(\alg{A})$ generates the same variety of Heyting algebras as $\alg{A}$. 
\end{proposition}

\begin{proof}
Let $\vty{V}(\alg{A})$ be the variety of Heyting algebras generated by $\alg{A}$. Since $\alg{A}$ is a Heyting subalgebra of $S(\alg{A})$, $\vty{V}(\alg{A})\subseteq\vty{V}(S(\alg{A}))$. But $S(\alg{A})$ is a Heyting subalgebra of the product $P(\alg{A})=\prod_Y\alg{A}_y$ and each $\alg{A}_y$ is a Heyting algebra quotient of $\alg{A}$, hence belongs to $\vty{V}(\alg{A})$. The other containment follows. 
\end{proof}

\begin{proposition} \label{g}
If $\alg{A}$ is a centrally supplemented Heyting algebra, then $A=S(\alg{A})$. In particular, $S(S(\alg{A}))=S(\alg{A})$. 
\end{proposition}

\begin{proof}
In Proposition~\ref{q} it was shown that if $a\in\alg{A}$ has a central supplement $a^+$, then considered as elements of the product we have that $a^+$ is the supplement of $a$ in the product as well. If each element of $\alg{A}$ has a central supplement, then $\alg{A}$ is a supplemented subalgebra of the product, hence is equal to $S(\alg{A})$. The further remark follows since $S(\alg{A})$ is centrally supplemented. 
\end{proof}

We next consider the center of $S(A)$. The key facts are given in the following. Here we consider for a Heyting algebra $A$ the relation $\theta_{\!A}$ consisting of those pairs $(c,d)$ such that for each $a\in A$ we have $a\vee c=1$ iff $a\vee d=1$. It is easily seen that $\theta_{\!A}$ is a lattice congruence on $A$. In fact, this is a lattice congruence on any distributive lattice, and has been studied in its order dual version by Speed in~\cite[Sec.~5]{Spe69A} and~\cite{Spe69B}. 

\begin{proposition} \label{Z(S(A))}
For a Heyting algebra $A$, the map $\psi \colon A\to Z(S(A))$ where $\psi(a)=\ind_{(a=1)}$ is a bounded lattice homomorphism with kernel $\theta_{\!A}$. The image of this map is $D(A)$ and this image generates $Z(S(A))$ as a Boolean algebra. 
\end{proposition}

\begin{proof}
In the proof of Proposition~\ref{D} it is shown that $D(A)$ is a bounded sublattice of the center of $S(A)$, and the proof shows that $\psi$ is a bounded lattice homomorphism. By Proposition~\ref{fr} an element $u\in S(A)$ has a representation as $\bigvee_{i=1}^na_i\wedge e_i$ for some finite partition of unity $e_1,\ldots,e_n$ of $B(A)$ and some $a_1,\ldots,a_n\in A$. So if $u$ is in the center of $S(A)$ we have $u=u^{++}$ hence $u=\bigvee_{i=1}^na_i^{++}\wedge e_i$. But $a_i^{++}=\ind_{(a=1)}$ so belongs to $D(A)$, and by definition each of $e_1,\ldots,e_n$ belongs to $B(A)$. So $u$ belongs to the Boolean subalgebra of the center generated by $S(A)$, hence $B(A)$ is the center of $S(A)$. 

It remains to show that $\theta_{\!A}$ is the kernel of $\psi$. Suppose $\psi(c)=\psi(d)$. Then $\ind_{(c=1)}=\ind_{(d=1)}$. Then for any $a\in A$ we have $a\vee c=1$ iff $\lbr a<1\rbr\subseteq \lbr c=1\rbr$ and $a\vee d = 1$ iff $\lbr a<1\rbr\subseteq\lbr d =1\rbr$ and it follows that $a\vee c = 1$ iff $a\vee d=1$, so $(c,d)\in\theta_{\!A}$. Conversely, suppose $(c,d)$ does not belong to the kernel of $\psi$. Without loss of generality there is $y\in Y$ with $c(y)=1$ and $d(y)<1$. Then $y$ belongs to the clopen set $N=\lbr c=1\rbr\cap\lbr d<1\rbr$. By Lemma~\ref{min}, the sets of the form $Y\setminus\widehat{a}=\lbr a<1\rbr$ are a basis for the topology on $Y$, so there is $a\in A$ with $\lbr a<1\rbr\subseteq N$. Then $a\vee c=1$ and $a\vee d<1$, giving that $(c,d)$ does not belong to $\theta_{\!A}$. 
\end{proof} 

We describe the dual Stone space of the center of $S(\alg{A})$. 

\begin{lemma} \label{lem:C(Y)}
Let $\alg{A}$ be a Heyting algebra with dual Esakia space $X$ and minimum $Y$, then 
\[
\C(Y) = \bigcap \{ \widehat{a} \cup X \setminus \widehat{b} : a,b \in A, \ a \lor b^+ = 1 \},  
\]
where $\C(Y)$ denotes the closure of $Y$ in $X$ and $b^+$ is computed in $S(\alg{A})$.    
\end{lemma} 

\begin{proof}
We claim that $\widehat{a} \cup X \setminus \widehat{b} \supseteq Y$ iff $a \lor b^+ = 1$. Since the sets of the form $\widehat{a} \cup X \setminus \widehat{b}$, with $a, b \in A$, form a basis for the closed sets of $X$ this will suffice to establish the lemma. Therefore, let $a, b \in A$ be given. Then, $\widehat{a} \cup X \setminus \widehat{b} \supseteq Y$ iff for all $y \in Y$ we have that $y \in \widehat{a}$ or $y \in  X \setminus \widehat{b}$. This is equivalent to having $a(y) = 1$ or $b(y) < 1$ for all $y \in Y$. Since $b^+ = \ind_{(b<1)}$ and the factors $\alg{A}_y$ are all fsi this in turn is equivalent to $a \lor b^+ = 1$, as desired.     
\end{proof}

\begin{proposition} \label{prop:C(Y)=Stone-center-Z(SA)}
Let $\alg{A}$ be a Heyting algebra with dual Esakia space $X$ and minimum $Y$, then  the dual Stone space of $Z(S(\alg{A}))$ is homeomorphic to $\C(Y)$.   
\end{proposition} 

\begin{proof}
Let $Z$ denote the dual Stone space of $Z(S(\alg{A}))$ and define a map $\gamma \colon Z \to X$ by letting $\gamma(z) = {\uparrow}z \cap A$. Using Lemma~\ref{min}, $\gamma$ is easily seen to be a well defined map with the property that $a \in \gamma(z)$ iff $a^{++} \in z$, for all $a \in A$. 

We first show that the image of $\gamma$ is $\C(Y)$. If $z \in Z$ and $a, b \in A$ are such that $a \lor b^+ =1$, then also $a^{++} \lor b^{+} = 1$. As both $a^{++}$ and $b^+$ belongs to the center of $S(\alg{A})$ we must have that $a^{++} \in z$ or $b^{+} \in z$. In the former case $a \in \gamma(z)$ and in the latter case that $b \not\in \gamma(z)$. From Lemma~\ref{lem:C(Y)} we may then conclude that $\gamma(z) \in \C(Y)$. Conversely, for $x \in \C(Y)$ consider the sets $F = {\uparrow}\{a^{++}: a \in x\}$ and $I ={\downarrow}\{b^{++} : b \not\in x\}$. It is easy to see that $F$ is a filter and $I$ an ideal of $Z(S(A))$. In fact, $F \cap I = \emptyset$. For suppose not, then we must have $a \in x$ and $b \not\in x$ such that $a^{++} \leq b^{++}$. But then $a^{++} \leq b$ and so $b \lor a^+ = 1$. As $x \not\in \widehat{b}$ and $x \not\in X\setminus \widehat{a}$ this contradicts Lemma~\ref{lem:C(Y)}. Thus there exists an ultrafilter $z$ on $Z(S(A))$ such that $F \subseteq z$ and $z \cap I = \emptyset$. We claim that $\gamma(z) = x$. If $a \in x$ then $a^{++} \in F \subseteq z$ whence $a \in \gamma(z)$. Conversely, if $a \not\in x$ then $a^{++} \in I$ whence $a^{++} \not\in z$ and so $a \not\in \gamma(z)$. Thus $\gamma(z)=x$.      

We then observe that $\gamma^{-1}(\widehat{a} \cap \C(Y)) = \widehat{a^{++}}$ and $\gamma^{-1}(\C(Y) \setminus \widehat{a}) = Z\setminus \widehat{a^{++}}$ for all $a \in A$. Consequently $\gamma$ is a continuous surjection between Stone spaces. Thus to establish the proposition it suffices to show that $\gamma$ is an injection. To this end assume that $\gamma(z) = \gamma(z')$ for some $z,z' \in Z$. Then we must have that $a^{++} \in z$ iff $a^{++} \in z'$ for all $a \in A$. But then also $a^{+} \in z$ iff $a^{+} \in z'$ for all $a \in A$. From Proposition~\ref{Z(S(A))} we may then conclude that $z = z'$.  
\end{proof}

\begin{lemma} \label{nm}
For a Heyting algebra $\alg{A}$ we have $\{a\wedge\ind_{(s<1)} : a,s\in\alg{A}\}$ is join\hyp{}dense in $S(\alg{A})$ and $\{(a\wedge\ind_{(s<1)})\vee\ind_{(s=1)} : a,s\in\alg{A}\}$ is meet\hyp{}dense in $S(\alg{A})$.  
\end{lemma}

\begin{proof}
Suppose $u\in S(\alg{A})$. By Proposition~\ref{fr}, $u=\bigvee_1^na_i\wedge e_i$ for some $a_1,\ldots,a_n\in A$ and $e_1,\ldots,e_n\in B(\alg{A})$. Each $e_i$ is the characteristic function of a clopen set $K_i$ of $Y$. Let $y\in K_i$. By Lemma~\ref{min} the sets $Y\setminus\widehat{s}$ are a basis for the topology. Therefore there is $s\in A$ with $y\in\lbr s<1\rbr\subseteq K_i$. It follows that $a\wedge\ind_{(s<1)}\leq u\leq (a\wedge\ind_{(s<1)})\vee\ind_{(s=1)}$ with all three functions agreeing at $y$. The result follows. 
\end{proof}

For algebras $\alg{A}\leq\alg{B}$, we say $\alg{B}$ is an \emph{essential extension} of $\alg{A}$ if every non\hyp{}trivial congruence of $\alg{B}$ restricts to a non\hyp{}trivial congruence of $\alg{A}$. When these are Heyting algebras, it is simple to see that this is equivalent to every $b<1$ in $\alg{B}$ having an $a<1$ in $\alg{A}$ with $b\leq a$. Clearly any meet\hyp{}dense extension of a Heyting algebra $\alg{A}$ is an essential extension, thus the MacNeille completion $\overline{\alg{A}}$ of $\alg{A}$ is an essential extension. 

\begin{proposition}\label{x}
For a Heyting algebra $\alg{A}$, the extension $\alg{A}\leq S(\alg{A})$ is essential. 
\end{proposition}

\begin{proof}
By Lemma~\ref{nm} it is enough to show that if $a,s\in\alg{A}$ and $u=(a\wedge\ind_{(s<1)})\vee\ind_{(s=1)}$ is strictly less than 1, then there is an element of $\alg{A}$ distinct from $1$ that lies above $u$. Clearly $a\vee s$ is an element of $\alg{A}$ that lies above $u$. Since $u<1$ there is $y\in Y$ with $u(y)<1$. This $y$ must be such that $a(y)<1$ and $s(y)<1$. Since $\alg{A}_y$ is fsi, then $(a\vee s)(y)<1$.
\end{proof}

An embedding $\alg{A}\leq\alg{B}$ is \emph{regular} if it preserves all existing joins and meets in $\alg{A}$. Join\hyp{}density implies meet\hyp{}regularity, and meet\hyp{}density implies join\hyp{}regularity. Thus the embedding of a lattice into its MacNeille completion is regular because it is both join\hyp{}dense and meet\hyp{}dense. The following notion, first introduced in~\cite{CGT17}, is key to determining when the embedding $\alg{A}\leq S(\alg{A})$ is regular. 

\begin{definition}
A Heyting algebra $\alg{A}$ is \emph{externally distributive} if whenever $S$ is a subset of $\alg{A}$ whose meet exists, for any $a\in\alg{A}$ with $a\vee s=1$ for all $s\in S$ we have $a\vee\bigwedge S = 1$. 
\end{definition}

\begin{proposition} \label{sd}
If $\alg{A}$ is an externally distributive Heyting algebra, then $\alg{A}\leq S(\alg{A})$ is a regular embedding. 
\end{proposition}

\begin{proof}
As in any Heyting algebra, if $\bigwedge_I c_i$ exists, then $a\to\bigwedge_Ic_i=\bigwedge_I(a\to c_i)$, and if $\bigvee_Ic_i$ exists, then $(\bigvee_Ic_i)\to a=\bigwedge_I(c_i\to a)$. 

Suppose $\bigwedge_Ic_i$ exists in $\alg{A}$ and is equal to $c$. By Lemma~\ref{nm}, to show that $c$ is also the meet of the family $\{c_i\}_{i\in I}$ in $S(\alg{A})$ it suffices to show that for $a,s\in\alg{A}$, if $a\wedge\ind_{(s<1)}$ is beneath $c_i$ for each $i\in I$, then it is beneath $c$. Note that $a\wedge\ind_{(s<1)} \leq c_i$ iff $s\vee (a\to c_i) = 1$. 
So if $a\wedge\ind_{(s<1)}\leq c_i$ for each $i\in I$, then $s\vee (a\to c_i) = 1$ for each $i\in I$. Since $\bigwedge_Ic_i = c$ we have $\bigwedge_I(a\to c_i)$ exists and is equal to $a\to c$. So by external distributivity, $s\vee (a\to c)=1$, and this implies $a\wedge\ind_{(s<1)}\leq c$ as required. 

Next, suppose $\bigvee_Ic_i$ exists in $A$ and is equal to $c$. To show that this join is preserved, by Lemma~\ref{nm}, it it is enough to show that for $a,s\in\alg{A}$, if $c_i\leq (a\wedge\ind_{(s<1)})\vee\ind_{(s=1)}$ for all $i\in I$ then $c\leq (a\wedge\ind_{(s<1)})\vee\ind_{(s=1)}$. But $c_i\leq (a\wedge\ind_{(s<1)})\vee\ind_{(s=1)}$ iff $s\vee (c_i\to a)=1$. Our assumption gives $s\vee(c_i\to a)=1$ for each $i\in I$. Since $\bigvee_Ic_i=c$ we have $\bigwedge_I(c_i\to a)$ exists and is equal to $c\to a$. So external distributivity gives $s\vee(c\to a)=1$, hence $c\leq (a\wedge\ind_{(s<1)})\vee\ind_{(s=1)}$ as required. 
\end{proof}

\begin{theorem} \label{r}
For a Heyting algebra $\alg{A}$, these are equivalent. 
\begin{enumerate}
\item The embedding of $\alg{A}$ into $S(\alg{A})$ is regular, 
\item The embedding of $\alg{A}$ into $S(\alg{A})$ is meet\hyp{}regular, 
\item $\alg{A}$ is externally distributive. 
\end{enumerate}
\end{theorem}

\begin{proof}
Proposition~\ref{sd} provides (3) $\Rightarrow$ (1), and (1) $\Rightarrow$ (2) is trivial. To see (2) $\Rightarrow$ (3) suppose $S\subseteq\alg{A}$ and $\bigwedge S$ exists and $a\in\alg{A}$ is such that $a\vee s=1$ for each $s\in S$. Then for each $y\in Y$, if $a(y)<1$ we have $s(y)=1$, so $\ind_{(a<1)}\leq s$ for each $s\in S$. Since the embedding is regular, the meet $\bigwedge S$ taken in $\alg{A}$ is also the meet in $S(\alg{A})$, and it follows that $\ind_{(a<1)}\leq\bigwedge S$. This implies that if $a(y)<1$, then $(\bigwedge S)(y)=1$, hence $a\vee\bigwedge S=1$. 
\end{proof}

We conclude this section by giving an abstract characterization of the centrally supplemented extensions $S(\alg{A})$ of a Heyting algebra $\alg{A}$. More precisely we show that the extension $S(\alg{A})$ is the free centrally supplemented extension of $\alg{A}$ with respect to the class of so\hyp{}called S\hyp{}homomorphisms. What follows is essentially the order dual version of a special case of Davey's construction of the $\mathfrak{m}$\hyp{}Stone extension of a bounded distributive lattice~\cite{Dav72}. 

For a bounded distributive lattice $\alg{D}$ we define the \emph{co\hyp{}annihilator}, or \emph{polar}, of an element $a \in D$ to be the set 
\[
a^\top = \{b \in D : a \lor b = 1\}.
\]
This is evidently the order dual of the well\hyp{}known notion of the annihilator of an element. 

\begin{definition}
A homomorphism $h \colon \alg{D} \to \alg{E}$ between bounded distributive lattice is called an \emph{S\hyp{}homomorphism} provided that $a^\top = b^\top$ implies $h(a)^\top = h(b)^\top$ for all $a, b \in D$. 
\end{definition}

\begin{proposition}
Any homomorphism $h \colon \alg{D} \to \alg{E}$ of supplemented distributive lattices is an S\hyp{}homomorphism.
\end{proposition}

\begin{proof}
We observe that in a supplemented distributive lattice the co\hyp{}annihilator of any element is the principal filter generated by its supplement. Consequently, if $a^\top = b^\top$, then $a^+ = b^+$ and hence 
\[
h(a)^\top = {\uparrow}h(a)^+ =  {\uparrow}h(a^+) = {\uparrow}h(b^+) = {\uparrow}h(b)^+ = h(b)^\top,
\]
showing that $h$ is an S\hyp{}homomorphism. 
\end{proof}

\begin{proposition}
Let $\alg{A}$ be a Heyting algebra. The embedding $\alg{A} \hookrightarrow S(\alg{A})$ is an S\hyp{}homomorphism. 
\end{proposition} 

\begin{proof}
Suppose that $a, b \in S(\alg{A})$ are such that $a^\top \neq b^\top$ as computed in $S(\alg{A})$. Then without loss of generality we have $a \lor u = 1$ and $b \lor u < 1$ for some $u \in S(\alg{A})$. As $S(\alg{A})$ supplemented and therefore externally distributive it follows from Lemma~\ref{nm} that there must be $c, s \in A$ with $u \leq (c \land \ind_{(s<1)}) \lor \ind_{(s=1)}$ and $b \lor (c \land \ind_{(s<1)}) \lor \ind_{(s=1)} < 1$. Since $(c \land \ind_{(s<1)}) \lor \ind_{(s=1)} \leq c \lor s$ we obtain that $a \lor c \lor s = 1$. On the other hand since $b \lor (c \land \ind_{(s<1)}) \lor \ind_{(s=1)} < 1$ we must have $y \in Y$ such that $b(y), c(y), s(y) <1$, so with $(b \lor c \lor s)(y) < 1$. So $b \lor c \lor s <1$. This shows that  $a^\top \neq b^\top$ also as computed in $\alg{A}$.  
\end{proof}

Recall from Propositions~\ref{lem:coreg} and~\ref{prop:deMorgan-eq-co-Stone} that if $\alg{E}$ is a centrally supplemented distributive lattice, then we have a map of bounded distributive lattices $\alg{E} \to Z(\alg{E})$ given by $e \mapsto e^{++}$.  Thus given any bounded lattice homomorphism $h \colon \alg{A} \to \alg{E}$ we obtain a bounded lattice homomorphism $h^{++} \colon \alg{A} \to Z(\alg{E})$ by composition. Further recall from Proposition~\ref{Z(S(A))} that we have an onto bounded lattice homomorphism $\psi \colon \alg{A} \to Z(S(\alg{A}))$ given by $\psi(a) = \ind_{(a =1)}$, the image of which generates the center $Z(S(\alg{A}))$.      

\begin{lemma}\label{dji}
Let $h\colon\alg{A}\to\alg{E}$ be an S-homomorphism of Heyting algebras with $E$ centrally supplemented. Then there is a Boolean algebra homomorphism $\widetilde{h}\colon Z(S(A))\to Z(E)$ with the following diagram commuting. This $\widetilde{h}$ is unique. 
\end{lemma}

\begin{center}
\begin{tikzpicture}
\node at (0,0) {$Z(S(A))$};
\node at (4,0) {$Z(E)$};
\node at (0,2) {$A$};
\node at (4,2) {$E$};
\draw[->] (1,0) -- (3,0);
\draw[->] (1,2) -- (3,2);
\draw[->] (0,1.5) -- (0,.5);
\draw[->] (4,1.5) -- (4,.5);
\draw[->] (1,1.5) -- (3,0.5);
\node at (2,2.5) {$h$};
\node at (-.5,1) {$\psi$};
\node at (5,1) {$(\cdot)^{++}$};
\node at (2,-.5) {$\widetilde{h}$};
\node at (2.5,1.3) {$h^{++}$};
\end{tikzpicture}
\end{center}

\begin{proof}
By the definition of $h^{++}$ the upper triangle commutes. We next show $\ker\psi\subseteq \ker h^{++}$. Let $(c,d) \in \ker\psi$. By Proposition~\ref{Z(S(A))}, $c^\top = d^\top$. Since $h$ is an S-homomorphism we have that $h(c)^\top = h(d)^\top$. Since $\alg{E}$ is supplemented we then have ${\uparrow}h(c)^+ = h(c)^\top = h(d)^\top =  {\uparrow}h(d)^+$. So $h(c)^+ = h(d)^+$, giving $h(c)^{++} = h(d)^{++}$, hence $(c,d)\in\ker h^{++}$. 

Having shown that $\ker\psi\subseteq \ker h^{++}$, it follows that there is a bounded distributive lattice homomorphism $f\colon \psi(A)\to Z(E)$ with $f(\psi(a))=h^{++}(a)$. By Proposition~\ref{Z(S(A))}, the image of $\psi$ generates the Boolean algebra $Z(S(A))$, so by \cite[p.~97]{BD74} $Z(S(A))$ is the free Boolean extension of $\psi(A)$. So there is a unique Boolean algebra homomorphism $\widetilde{h}\colon Z(S(A))\to Z(E)$ that extends $f$. This establishes our result. 
\end{proof}

\begin{theorem}[{cf.~\cite[Thm.~3]{Dav72}}]
Suppose $h \colon \alg{A} \to \alg{E}$ is an S\hyp{}homomorphism of Heyting algebras with $\alg{E}$ centrally supplemented. Then $h$ has a unique extension to a homomorphism $\overline{h} \colon S(\alg{A}) \to \alg{E}$ of supplemented Heyting algebras.      
\end{theorem}  

\begin{proof}
We first show that for $a,b\in A$ and $e\in Z(S(A))$ 
\begin{equation}\label{kl}
a\wedge e=b\wedge e\quad\Rightarrow\quad h(a)\wedge\widetilde{h}(e)=h(b)\wedge\widetilde{h}(e)
\end{equation}
Viewing $A$ and $S(A)$ as subalgebras of the product $P(A)$, the assumption gives that $a,b$ agree on the set $\lbr e = 1\rbr$. Therefore $e\leq a\leftrightarrow b$. This gives $e\leq\psi(a\leftrightarrow b)$, hence $\widetilde{h}(e)\leq\widetilde{h}(\psi(a\leftrightarrow b))$. By Lemma~\ref{dji}, $\widetilde{h}(e)\leq (h(a)\leftrightarrow h(b))^{++}\leq h(a)\leftrightarrow h(b)$. Since $\widetilde{h}(e)$ is central in $E$, this gives $h(a)\wedge \widetilde{h}(e)=h(b)\wedge\widetilde{h}(e)$. 

A normal expression is one of the form $(a_1 \land e_1) \lor \ldots \lor (a_n \land e_n)$ where $a_1, \ldots, a_n \in A$ and $e_1, \ldots, e_n \in Z(S(\alg{A}))$ is a partition of unity. By Proposition~\ref{fr} each element of $S(\alg{A})$ is determined by a normal expression. Thus we may define a function $\overline{h}$ from $S(\alg{A})$ to $\alg{E}$ as by letting
\[{\textstyle 
\overline{h}\left(\,\bigvee_{i=1}^n a_i\wedge e_i \right) = \bigvee_{i=1}^n h(a_i)\wedge \widetilde{h}(e_i) }.
\]
We must show this is indeed well defined. Once this is shown, since $a\wedge 1$ is a normal expression for $a$, it follows that $\overline{h}(a) = h(a)$, hence $\overline{h}$ is an extension of $h$. 

Suppose $u\in S(A)$. If $\bigvee_{i=1}^n a_i\wedge e_i$ and $\bigvee_{i=1}^n b_i\wedge e_i$ are two normal expressions for $u$ using the same partition of unity. Then for each $i$ we have $a_i\wedge e_i=b_i\wedge e_i$, and it follows from (\ref{kl}) that $\overline{h}$ agrees for these two normal expressions. A partition of unity that refines $e_1,\ldots,e_n$ is of the form $(f_{i,j})$ where $i=1,\ldots,n$ and for each $i$ we have $j=1,\ldots,n_i$ and $\bigvee_{j=1}^{n_i}f_{i,j}=e_i$. Then $\bigvee_{i,j} a_i\wedge f_{i,j}$ is a normal expression for $u$. Using distributivity and the fact that $\widetilde{h}$ is a homomorphism, for each $i$ we have $\bigvee_1^{n_i}h(a)\wedge\widetilde{h}(f_{i,j}) = h(a_i)\wedge\widetilde{h}(e_i)$. It follows that $\overline{h}$ takes the same value on $\bigvee_{i=1}^na_i\wedge e_i$ and $\bigvee_{i,j}a_i\wedge f_{i,j}$. From what has been shown, we know that $\overline{h}$ takes the same value on all normal expressions for $u$ based on the partition $e_1,\ldots,e_n$ or any partition that refines this. Since any two partitions of unity have a common refinement, it follows that $\overline{h}$ takes the same value on all normal expressions for $u$, hence is well defined. 

To see that $\overline{h}$ is a Heyting algebra homomorphism, let $u,v\in S(A)$. By Proposition~\ref{fr} these elements have normal expressions over a common partition of unity, say $u=\bigvee_{i=1}^n a_i\wedge e_i$ and $v=\bigvee_{i=1}^nb_i\wedge e_i$. Then 
\begin{align*}
	\overline{h}(u)&={\textstyle \bigvee_{i=1}^n h(a_i)\wedge\widetilde{h}(e_i)}\\
	\overline{h}(v)&={\textstyle \bigvee_{i=1}^n h(b_i)\wedge\widetilde{h}(e_i)}
\end{align*}
	Since $\widetilde{h}$ is a Boolean homomorphism, we have that $\widetilde{h}(e_1),\ldots\widetilde{h}(e_n)$ is a partition of unity of $Z(E)$, and it follows that $h(u)\wedge h(v) = \bigvee_{i=1}^n h(a_i\wedge b_i)\wedge\widetilde{h}(e_i)$, hence that $\overline{h}$ preserves finite meets. A similar argument shows that $\overline{h}$ preserves finite joins and Heyting implication. 
To see that $\overline{h}$ preserves supplements, let $u$ have normal expression $\bigvee_{i=1}^na_i\wedge e_i$. Since $S(A)$ is centrally supplemented, its supplement satisfies both De Morgan laws, so $u^+=\bigwedge_{i=1}^n a_i^+\vee e_i^+$. Then using the facts that $\overline{h}$ is a Heyting homomorphism and $E$ is centrally supplemented, so its supplement satisfies both De Morgan laws, we have 
\begin{align*}
	\overline{h}(u^+)&={\textstyle \bigwedge_{i=1}^n \overline{h}(a_i^+)\vee\overline{h}(e_i^+)}\\
	\overline{h}(u)^+&={\textstyle \bigwedge_{i=1}^n\overline{h}(a_i)^+\vee\overline{h}(e_i)^+}
\end{align*} 
Therefore, it is enough to show $\overline{h}(a^+)=\overline{h}(a)^+$ and $\overline{h}(e^+)=\overline{h}(e)^+$ for $a\in A$ and $e\in Z(S(A))$. Since $e=(1\wedge e)\vee(0\wedge e')$ is a normal expression for $e$, we have that $\overline{h}(e)=\widetilde{h}(e)$. So $\overline{h}$ restricts to $\widetilde{h}$ on the center of $S(A)$. Then since supplements of central elements are their complements and $\widetilde{h}$ is a Boolean homomorphism, $\overline{h}(e)^+=\widetilde{h}(e)^+=\widetilde{h}(e^+)=\overline{h}(e^+)$. For $a\in A$, by Lemma~\ref{dji} we have $h(a)^{++}=\widetilde{h}(a^{++})$. Therefore $\overline{h}(a)^+=h(a)^+=h(a)^{+++}=\widetilde{h}(a^{++})^+=\widetilde{h}(a^{+++})=\widetilde{h}(a^+)=\overline{h}(a^+)$, where the last equality is because $a^+$ is central. 

Finally, uniqueness of $\overline{h}$ follows as it is determined on a set $A\cup Z(S(A))$ that generates $S(A)$. 
\end{proof}

\begin{remark}
It is not the case that $S(\alg{A})$ is what one would call a free centrally supplemented Heyting extension of $\alg{A}$. One can see this by considering the only Heyting algebra embedding $h \colon \mathbf{3}\to\mathbf{2}\times\mathbf{3}$ where $\mathbf{2}$ and $\mathbf{3}$ are the 2\hyp{}element and 3\hyp{}element Heyting algebras, respectively. We have $S(\mathbf{3})=\mathbf{3}$, and $\mathbf{2}\times\mathbf{3}$ is centrally supplemented, but $h$ does not preserve supplements. 
\end{remark}

\section{hyper\hyp{}MacNeille completions}\label{hyper-MacNeille} \label{sec:HHMC}

The hyper\hyp{}MacNeille completion is a recent notion introduced in the general setting of pointed residuated lattices, also known as FL\hyp{}algebras~\cite[Sec.~6.2]{CGT17}. We are interested in its specialization to the Heyting algebra setting. Things simplify in this setting in several ways. First, there is a known ``frame\hyp{}theoretic'' simplification~\cite{Ter13}, and the purpose of this note is to provide a further algebraic simplification. We begin with the bare bones simplified description, and then relate this to frame\hyp{}theoretic versions in the following section. 

\begin{definition} \label{def:k}
For $\alg{A}$ a Heyting algebra, let $\alg{W}=\alg{A}^2$ and let $N$ be the binary relation on $\alg{W}$ given by
$(s,a)\, N\, (t,b)$ iff $s\vee t\vee (a\to b) = 1$. We then let $U,L$ be the  associated Galois connection
\begin{center}
\begin{tikzpicture}
\node at (0,0) {$\mathcal{P}(W)$};
\node at (4,0) {$\mathcal{P}(W)$};
\draw[->] (1,.1) -- (3,.1);
\draw[->] (3,-.1) -- (1,-.1);
\node at (2,.4) {$U$};
\node at (2,-.4) {$L$};
\end{tikzpicture}
\end{center}
where $U(X) = \{u \in W : \forall w \in X \ (w N u) \}$ and $L(Y) = \{w \in W: \forall u \in Y \ (w N u) \}$. 
\end{definition}

Let $W_{\!\alg{A}}$ be the polarity $(W,W,N)$ described above and let $\mathcal{G}(W_{\!\alg{A}})$ be its Galois closed elements, that is, those $X\subseteq W$ with $X=LU(X)$. As is the case with any Galois connection, these Galois closed subsets form a complete lattice under set inclusion where $\bigwedge_IX_i = \bigcap_IX_i$ and $\bigvee_IX_i = LU(\bigcup_IX_i)$. A further general property of Galois connections applied to this particular situation gives the following. 

\begin{lemma}\label{n}
For a Heyting algebra $\alg{A}$, the set $\{L(t,b):t,b\in\alg{A}\}$ is meet-dense in $\mathcal{G}(W_{\!\alg{A}})$ and the set $\{LU(s,a):s,a\in\alg{A}\}$ is join-dense in $\mathcal{G}(W_{\!A})$. 
\end{lemma}

Note that while we are using $L$ and $U$ for these operations, they do not signify lower and upper bounds. Indeed, the relation $N$ need not be a partial order. It is the case that $N$ is reflexive, but it need not be transitive nor anti-symmetric. A most convenient way to visualize this relation is through the centrally supplemented extension $S(\alg{A})$. We require notation.

\begin{definition}
For $a,b,s,t\in\alg{A}$ let $f(s,a)=a\wedge\ind_{(s<1)}$ and $g(t,b)=(b\wedge\ind_{(t<1)})\vee\ind_{(t=1)}$. Then set $F=\{f(s,a):a,s\in\alg{A}\}$ and $G=\{g(t,b):b,t\in\alg{A}\}$. 
\end{definition}

We note that Lemma~\ref{nm} says that $F$ is join\hyp{}dense in $S(\alg{A})$ and $G$ is meet\hyp{}dense. One may get the impression that there is a bijective correspondence between pairs $(s,a)$ and functions $f(s,a)$. This is not the case, many pairs $(s,a)$ can yield the same function $f(s,a)$. The crucial point in connecting $S(\alg{A})$ to the relation $N$ is the following simple observation. 

\begin{proposition}\label{c}
$(s,a)\,N\,(t,b)$ iff $f(s,a)\leq g(t,b)$.
\end{proposition}

\begin{proof}
Let $\alg{A}$ be a Heyting algebra and $a,b,s,t\in\alg{A}$. Consider $\alg{A}\leq\prod_Y\alg{A}_y$ as a subdirect product as in the previous section. Then $(s,a)N(t,b)$ iff $s\vee t\vee (a\to b)=1$. Since all of the factors $\alg{A}_y$ are fsi, this join is equal to 1 iff for each $y\in Y$ with $s(y)<1$ and $t(y)<1$ we have that $a(y)\leq b(y)$, and this occurs iff the indicated inequality between functions occurs. 
\end{proof}

\begin{theorem}\label{ggg}
There is an order embedding $\Delta \colon S(\alg{A})\to \mathcal{G}(W_{\!\alg{A}})$ given by 
\[
\Delta(u) = \{(s,a) : f(s,a)\,\leq \, u\}.
\]
This embedding is join\hyp{}dense and meet\hyp{}dense, so $\mathcal{G}(W_{\!\alg{A}})$ is the MacNeille completion of $S(\alg{A})$. 
\end{theorem}

\begin{proof}
For $u\in S(\alg{A})$, let ${\downarrow}u$ be the set of lower bounds of $u$ in $S(\alg{A})$ and ${\uparrow}u$ be the set of upper bounds. Since $F$ and $G$ are join\hyp{}dense and meet\hyp{}dense, respectively, it follows that $u$ is the join of ${\downarrow}u \cap F$ and $u$ is the meet of ${\uparrow}u\cap G$. It follows from this discussion and Proposition~\ref{c} that $U(\Delta(u))$ is the set of all $(t,b)$ such that $g(t,b)$ is an upper bound of $u$, and $LU(\Delta(u))$ is the set of all $(s,a)$ such that $f(s,a)\leq u$. So $\Delta(u)=LU(\Delta(u))$. 

Clearly $u\leq v$ implies $\Delta(u)\subseteq\Delta(v)$. But $\Delta(u)\subseteq\Delta(v)$ implies ${\downarrow}u\cap F\subseteq{\downarrow}v\cap F$, so $u\leq v$. Thus $\Delta$ is an order embedding. For $b,t\in\alg{A}$ we have $\Delta(g(t,b))=L(t,b)$ by Proposition~\ref{c}. For $a,s\in\alg{A}$, Proposition~\ref{c} gives $U(s,a)=\{(t,b):f(s,a)\leq g(t,b)\}$ and therefore $\Delta(f(s,a)) \subseteq LU(s,a)$. For the other containment, since $(s,a) \in \Delta(f(s,a))$ we have that $LU(s,a) \subseteq LU(\Delta(f(s,a))) = \Delta(f(s,a))$. Consequently, $\Delta(f(s,a))=LU(s,a)$. By Lemma~\ref{n}, $\{L(t,b):b,t\in\alg{A}\}$ and $\{LU(s,a):a,s\in\alg{A}\}$ are join\hyp{}dense and meet\hyp{}dense in $\mathcal{G}(W_{\!A})$ respectively. Thus $\Delta$ is a join\hyp{}dense and meet\hyp{}dense embedding. Since $\mathcal{G}(W_{\!A})$ is complete, it is the MacNeille completion of $S(\alg{A})$. 
\end{proof}

\begin{definition}
We call $\mathcal{G}(W_{\!A})$ the \emph{hyper\hyp{}MacNeille completion} of $\alg{A}$ and denote it $\alg{A}^+$. 
\end{definition}

We can easily accumulate a number of properties of the hyper\hyp{}MacNeille completion. 

\begin{theorem} \label{j}
For a Heyting algebra $\alg{A}$, 
\begin{enumerate}
\item $\alg{A}^+$ is a Heyting algebra,
\item $\alg{A}^+$ is centrally supplemented,
\item if $\alg{A}$ is centrally supplemented, $\alg{A}^+$ is the MacNeille completion of $\alg{A}$, 
\item if $\alg{A}$ is fsi, $\alg{A}^+$ is the MacNeille completion of $\alg{A}$, 
\item $\alg{A}^{++}$ is isomorphic to $\alg{A}^+$,
\item the center of $\alg{A}^+$ is complete and a complete sublatice of $\alg{A}^+$,
\item the embedding of $\alg{A}$ into $\alg{A}^+$ is regular iff $\alg{A}$ is externally distributive,
\item the embedding of $\alg{A}$ into $\alg{A}^+$ is essential, 
\item the embedding of $\alg{A}$ into $\alg{A}^+$ preserves central supplements,
\item $A^+$ is equal to $S(A)^+$,
\item $Z(A^+)$ is equal to $\overline{Z(S(A))}$ and is a complete subalgebra of $A^+$,
\item $Z(A^+)$ is the MacNeille completion of the free Boolean extension of $A/\theta_{\!A} \cong D(A)$,
\item if $A$ is complete, then $A^+=A$ iff $A$ is centrally supplemented. 
\end{enumerate}

\end{theorem}

\begin{proof}
(1) The hyper\hyp{}MacNeille completion $\alg{A}^+$ is the MacNeille completion of the Heyting algebra $S(\alg{A})$ and the MacNeille completion of a Heyting algebra is a Heyting algebra~\cite[Thm.~2.3]{HB04}. (2) The algebra $S(\alg{A})$ is supplemented by definition, so its MacNeille completion $\alg{A}^+$ is centrally supplemented by Proposition~\ref{prop:MN-of-DM-is-DM}. (3) This follows since $\alg{A}$ being centrally supplemented implies $A=S(\alg{A})$ by Proposition~\ref{g}. (4) This follows from (3) since a fsi Heyting algebra is centrally supplemented. (5) By (2) $\alg{A}^+$ is centrally supplemented, so by Proposition~\ref{g} we have $S(\alg{A}^+)=A^+$, so $A^{++}$ is the MacNeille completion of $A^+$ which is $A^+$ since $A^+$ is complete. (6) This follows from Proposition~\ref{prop:completecenter}. (7) This follows from the fact that the embedding of $A$ into $S(A)$ is regular iff $A$ is externally distributive given in Theorem~\ref{r} and the fact that the embedding of $S(A)$ into its MacNeille completion $A^+$ is always regular. (8) By Proposition~\ref{x}, the embedding of $A$ into $S(A)$ is essential, and the embedding of $S(A)$ into its MacNeille completion $A^+$ is essential. (9) That the embedding of $A$ into $S(A)$ preserves existing central supplements is given by Proposition~\ref{q}, and Proposition~\ref{prop:MN-of-DM-is-DM} shows that the embedding of the centrally supplemented algebra $S(A)$ into its MacNeille completion $A^+$ preserves central supplements. (10) This follows since $S(A)^+$ is the MacNeille completion of $S(S(A))$ and $S(S(A))$ is equal to $S(A)$ by Proposition~\ref{g}. (11) This follows from Proposition~\ref{center} and the facts that $S(A)$ is centrally supplemented and $A^+=\overline{S(A)}$. (12) This follows from (11), Proposition~\ref{center}, and the fact that any Boolean algebra generated by a bounded distributive lattice is its free Boolean extension. (13) If $A$ is centrally supplemented, then by Proposition~\ref{g}, $A=S(A)$, and then if $A$ is complete $A=\overline{S(A)}=A^+$. Conversely, if $A=A^+$, then $A$ is centrally supplemented by (2).
\end{proof}

\begin{remark}
In the wider setting of FL\hyp{}algebras a few results related to Theorem~\ref{j} are known. Namely, the hyper\hyp{}MacNeille completion and the MacNeille completion coincide for of any subdirectly irreducible FL\hyp{}algebra~\cite[Prop.~6.6]{CGT17}, and the embedding of an externally distributive FL\hyp{}algebra into its hyper\hyp{}MacNeille completion is regular~\cite[Thm.~6.11]{CGT17}. 
\end{remark}

\section{Other realizations of the hyper\hyp{}MacNeille completion} \label{other}

The previous section describes the hyper\hyp{}MacNeille completion of a Heyting algebra as the MacNeille of its centrally supplemented extension. This provides a convenient method to work with hyper\hyp{}MacNeille completions of Heyting algebras but is not its original source. As described in the introduction, hyper\hyp{}MacNeille completions arose in the study of proof theory for substructural logics. In~\cite{CGT17} the notion of a \emph{residuated hyper\hyp{}frame} which may be identified with a special class of so\hyp{}called \emph{residuated frames}~\cite{CGT12,GJ13}. It was shown that from a (pointed) residuated lattice one can construct a certain residuated hyper\hyp{}frame, and from this hyper\hyp{}frame construct a complete (pointed) residuated lattice extending the original. This was termed the hyper\hyp{}MacNeille completion. In this section we show that the notion of hyper\hyp{}MacNeille completion introduced in the previous section coincides with the original notion introduced by Ciabattoni, Galatos, and Terui. The results in this section are not needed in the remainder of the paper.   

Residuated frames are polarities equipped with additional structure ensuring that the lattice of Galois closed sets is a (pointed) residuated lattice. Since we are here only interested in the case of Heyting algebras, we do not need  it is not the notion of residuated frames in full generality, and it suffices to work with the simpler notion of a Heyting frame~\cite[Sec.~6.1]{Ter18}.

\begin{definition} \label{frame}
A \emph{Heyting frame} $\alg{W} = (W_0, W_1, N, \circ, \epsilon, \rightsquigarrow)$ is a structure consisting of a relation $N \subseteq W_0 \times W_1$, a monoid $(W_0, \circ, \epsilon)$, and an operation $\rightsquigarrow \colon W_0 \times W_1 \to W_1$ satisfying:

\begin{enumerate}
\item $w\circ v N u \iff v N w \rightsquigarrow u$,
\item $w \circ w N u \implies w N u$,
\item $\epsilon N u \implies w N u$,
\item $w \circ v N u \implies v \circ w N u$,
\end{enumerate}
for all $w, v \in W_0$ and $u \in W_1$. 
\end{definition}

The relation $N$ of a Heyting frame gives a Galois connection between $\mathcal{P}(W_0)$ and $\mathcal{P}(W_1)$. The purpose of the additional structure of a Heyting frame is to ensure that the lattice of Galois closed sets gives a Heyting algebra. The following is given in~\cite[Lem.~13]{Ter18}.

\begin{theorem}\label{thm:complete-residuated-lattices}
If $\alg{W} = (W_0, W_1, N, \circ, \epsilon, \rightsquigarrow)$ is a Heyting frame then the Galois closed subsets of $W_0$ form a complete Heyting algebra $\alg{W}^+$ with implication  
\[
X \to Y = \{w \in W_0 : \forall v \in X \ (v \circ w \in Y) \}. 
\]
\end{theorem} 

\begin{example}
For a Heyting algebra $\alg{A}$ let $\alg{M}_{\!\alg{A}} = (\alg{A}, \alg{A}, \leq, \land, 1, \to)$. It is not difficult to show that $\alg{M}_{\!\alg{A}}$ is a Heyting frame and $\alg{M}_{\!\alg{A}}^+$ is the MacNeille completion $\overline{\alg{A}}$ of $\alg{A}$.
\end{example}

We turn to a different application of Heyting frames that leads to the hyper\hyp{}MacNeille completion of a Heyting algebra. The approach is essentially the one presented in~\cite{Ter13}. For a Heyting algebra $\alg{A}$, we use $\alg{A}_\vee=(A,\vee,0)$ for its join\hyp{}semilattice reduct and $\alg{A}_\wedge=(A,\wedge,1)$ for its meet\hyp{}semilattice reduct. 

\begin{definition}
For a Heyting algebra $\alg{A}$ let $\alg{W}_{\!\alg{A}} = (W, W, N,\circ, (0, 1), \rightsquigarrow)$. Here the monoid $W = \alg{A}_\lor \times \alg{A}_\land$ has operation $(s,a)\circ(t,b)=(s\vee t,a\wedge b)$, the relation $N$ is given by $(s,a) N (t, b)$ iff $s \lor t \lor (a \to b) = 1$, and $(s, a) \rightsquigarrow (t, b) = (s \lor t, a \to b)$.
\end{definition} 

We now connect Heyting frames to the results of the previous section. 

\begin{proposition}
For $A$ a Heyting algebra, $\alg{W}_{\!\alg{A}}$ is a Heyting frame and the complete Heyting algebra $W_{\!A}^+$ of its Galois closed elements is the hyper\hyp{}MacNeille completion $A^+$ and this is isomorphic to the MacNeille completion of $S(\alg{A})$. 
\end{proposition}

\begin{proof}
It is straightforward to verify that $\alg{W}_{\!\alg{A}}$ is a Heyting frame. Note that the set $W$ and relation $N$ are exactly as in Definition~\ref{def:k}, so $\mathcal{G}(W_{\!A})=W_{\!A}^+$. The result follows from results of the previous section. 
\end{proof}

We next recall the original definition from~\cite[Sec.~6.2]{CGT17} of the hyper\hyp{}MacNeille completion of a Heyting algebra $\alg{A}$. Let $\alg{F}(A^2)$ be the free commutative monoid on the set $A^2$. We then obtain a monoid $\alg{M}_{\!\alg{A}}^\ast = \alg{F}(A^2) \times \alg{A}_\land$, with operation $(h, a) \circ (g, b) = (hg, a \land b)$ and unit $(\epsilon, 1)$, where $\epsilon$ denotes the empty word. The universal property of $\alg{F}(A^2)$ yields a monoid homomorphism $(\cdot)^\ast \colon \alg{F}(A^2) \to \alg{A}_\lor$ by letting $(a,b)^\ast = a \to b$.

\begin{definition}
For a Heyting algebra $\alg{A}$, let $\alg{V}_{\!\alg{A}} = (\alg{M}_{\!\alg{A}}^*, \alg{M}_{\!\alg{A}}^*, Q, \circ, (\epsilon, 1), \rightsquigarrow)$ where $Q$ is given by $(h,a) \mathrel{Q} (g, b)$ iff $h^\ast \lor g^\ast \lor (a \to b) = 1$, and $(h, a) \rightsquigarrow (g, b) = (hg, a \to b)$.
\end{definition}

It is not hard to show that $\alg{V}_{\!\alg{A}}$ is a Heyting frame. As establish in~\cite[Thm.~5.20]{CGT17} the resulting algebra of Galois closed sets is a completion of $\alg{A}$. 

\begin{theorem}
Let $\alg{A}$ be a Heyting algebra. Then there is a Heyting algebra embedding $\alpha \colon \alg{A}\to\alg{V}_{\!\alg{A}}^+$ given by $\alpha(b)=L(\epsilon,b)$. 
\end{theorem}

In~\cite[Sec.~6.2]{CGT17} the hyper\hyp{}MacNeille completion of $\alg{A}$ is defined to be the algebra $\alg{V}_{\!\alg{A}}^+$. We show that the Heyting algebra of Galois closed sets of $V_{\!\alg{A}}$ is the same as that of $W_{\!\alg{A}}$, which we know to be the Heyting algebra given by the MacNeille completion of $S(A)$. We leave to the reader the simple proof of the following. 

\begin{lemma} \label{d}
Let $R$ be a binary relation on a set $X$ and $S$ be a binary relation on a set $Y$. If $f \colon X\to Y$ is a surjective map with $u \mathrel{R} v$ iff $f(u) \mathrel{S} f(v)$, then the direct\hyp{}image function $f[\cdot] \colon \mathcal{G}(X)\to\mathcal{G}(Y)$ is an order\hyp{}isomorphism between their Galois closed sets. 
\end{lemma}

Using this lemma, we next show that the two definitions of hyper\hyp{}MacNeille completions coincide. 

\begin{theorem}
For $\alg{A}$ is a Heyting algebra, $\alg{W}_{\!\alg{A}}^+$ is isomorphic to $\alg{V}_{\!\alg{A}}^+$.
\end{theorem}

\begin{proof}
Note that there is a surjective map $f \colon V_{\!\alg{A}}\to W_{\!\alg{A}}$ given by $f(h,a)=(h^\ast,a)$ and $(h,a) \mathrel{Q} (g,b)$ iff $(h^\ast,a) \mathrel{N} (g^\ast,b)$. The result follows from Lemma~\ref{d}.
\end{proof}

\section{Boolean products and related sheaves} \label{sheaf}

There is a long history of presenting algebras by the global sections of a sheaf. Recall that a sheaf consists of a local homeomorphism $\pi\colon S\to Y$ between topological spaces $S$, called the sheaf space, and $Y$, the base space. A global section of this sheaf is a continuous map $\sigma\colon Y\to S$ with $\pi\circ\sigma$ being the identity on $Y$. Let $\Gamma(S)$ be the set of global section of $S$. If for each $y\in Y$ the stalk $S_y=\pi^{-1}[\{y\}]$ carries the structure of an algebra of a given type, and the topology of $S$ is compatible with the induced algebraic operations, then the global sections of the sheaf form a subalgebra of the product $\prod_YS_y$. Instances of this situation are weak Boolean products. Here we follow~\cite{BS81,BW79,CHJ96}. 

\begin{definition}
A subdirect product $\alg{A} \leq \prod_{Y} \alg{A}_y$ of a family of algebras is a \emph{weak Boolean product} if the index set $Y$ is equipped with a Stone space topology and for all $a,b\in A$: 
\begin{enumerate}
\item $\llbracket a = b \rrbracket = \{y \in Y : a(y) = b(y) \}$ is open,
\item If $N \subseteq Y$ is clopen, there is $c \in A$ with $N=\llbracket c = a \rrbracket$ and $Y \backslash N = \llbracket c = b \rrbracket$. 
\end{enumerate}
The first property is called \emph{equalizers are open}, and the second is the \emph{patchwork property}. If moreover for each $a, b \in A$ the set $\llbracket a = b \rrbracket$ is clopen then we say this $\alg{A}$ is a \emph{Boolean product}. 
\end{definition}

Each weak Boolean product gives rise to a sheaf representation. In the following, we make the harmless assumption that in a weak Boolean product $A\leq\prod_YA_y$ the stalks $A_y$ are pairwise disjoint. A notationally more cumbersome approach ``disjointifies'' the stalks if this is not the case. The following is well known~\cite{CHJ96}. 

\begin{proposition}
Let $\alg{A}\leq \prod_Y\alg{A}_y$ be a weak Boolean product. Set $S$ to be the disjoint union of the stalks $\alg{A}_y$ for $y\in Y$, and for each $a\in A$ and open $U\subseteq Y$ set $\mathcal{O}(a,U)=\{a(y):y\in U\}$. Then the sets of the form $\mathcal{O}(a,U)$ where $a\in A$ and $U\subseteq Y$ is open form a base for a topology on $S$ making the operations of $S$ continuous. The natural map $\pi \colon S\to Y$ is a local homeomorphism, and $\alg{A}$ is the algebra $\Gamma(S)$ of global sections of this sheaf. 
\end{proposition}

There is a standard method to construct a weak Boolean product representation of any algebra $\alg{A}$ via its factor congruences~\cite{Com71,Dav73}. This comes from the Pierce sheaf representation of a ring in terms of its central idempotents~\cite{Pie67}. We describe things in the setting of Heyting algebras. For a central element $c$ in a Heyting algebra $\alg{A}$ there is a congruence $\theta_c$, given by 
\[ 
a \mathrel{\theta_c} b \quad \iff \quad a\wedge c = b\wedge c.
\]
This is not only a Heyting algebra congruence, but is compatible with existing supplements too. For $X$ the dual Stone space of $Z(\alg{A})$ and an ultrafilter $x\in X$ let $\theta_x$ be the up\hyp{}directed union of $\{\theta_c:c\in x\}$. This also is a congruence compatible with existing supplements. 

\begin{proposition}
For a Heyting algebra $\alg{A}$, the subdirect representation $\alg{A}\leq \prod_X\alg{A}_x$ over the Stone space of its center is a weak Boolean product representation. 
\end{proposition}

Following~\cite{CHJ96} we refer to this the \emph{usual representation of $\alg{A}$ over its center}. We call the sheaf associated to this weak Boolean product representation \emph{the central sheaf} of $A$. 

\begin{proposition}\label{prop:stalks-of-S(A)}
Let $\alg{A}$ be a Heyting algebra. The stalks of the central sheaf of $S(\alg{A})$ are of the form $\alg{A}/\theta_x$ for $x \in \C(Y)$, the closure of the minimum $Y$ in $X_\alg{A}$.    
\end{proposition}

\begin{proof}
The stalks of the central sheaf of $S(\alg{A})$ are of the form $S(\alg{A})/\theta_z$ for $z$ a ultrafilter over the center of $S(\alg{A})$. We first show that the induced map $\beta_z \colon \alg{A} \to S(\alg{A})/\theta_z$ is surjective from which it follows that $\alg{A}/\!\ker\beta_z \cong S(\alg{A})/\theta_z$. To this end let $u \in S(\alg{A})$ be given. By Proposition~\ref{fr}, $u = \bigvee_{i = 1}^n a_i \land e_i$ for some $a_1, \ldots, a_n \in A$ and some central partition of unity $e_1, \ldots, e_n$ in $S(\alg{A})$. We must have $e_j \in z$ for some $1 \leq j \leq n$. But $u \land e_j = a_j \land e_j$, whence $u/\theta_z = a_j/\theta_z$, showing that $\beta_z(a_j) = u/\theta_z$. 

We next show that $\ker\beta_z = \theta_x$ for some $x \in \C(Y)$. For this it suffices to show that the equivalence class of $1$ under the congruence $\ker\beta_z$ is a prime filter belonging to $\C(Y)$. Evidently, $(a, 1) \in \ker\beta_z$ iff $a \geq c$ for some $c \in z$. Hence the equivalence class of $1$ is the set $x = {\uparrow}z \cap A$, which by Proposition~\ref{prop:C(Y)=Stone-center-Z(SA)} is indeed a prime filter in the closure of $Y$. 
\end{proof}

The following is taken from~\cite{CHJ96}. It applies in a more general setting, but we apply it only to Heyting algebras. 

\begin{definition}
A Heyting algebra is \emph{Hausdorff} if the usual representation over its center is a Boolean product. 
\end{definition} 

In~\cite[Thm.~9.5]{Vag95} centrally supplemented Heyting algebras were shown to be exactly the Boolean products of fsi Heyting algebras. In fact something slightly stronger can be said. 

\begin{proposition} \label{lk}
A Heyting algebra $\alg{A}$ is centrally supplemented iff it is Hausdorff and its central sheaf has fsi stalks. 
\end{proposition}

\begin{proof}
If $\alg{A}$ is centrally supplemented, then by Lemma~\ref{min} the filters generated by the ultrafilters of its center are the minimal prime filters of $A$. Thus the stalks of its central sheaf are the $\alg{A}_x$ where $x$ is a minimal prime filter, and therefore are fsi. For $a,b\in A$ we have $a(x)=b(x)$ iff $(a\leftrightarrow b)(x)=1$. So to show equalizers are clopen, it is enough to show that for each $a\in A$ that $N=\lbr a=1\rbr$ is clopen. Since the stalks are fsi we have $M=\lbr a^+=1\rbr$ contains the complement of $N$. If $M$ and $N$ are not disjoint there is a non\hyp{}empty clopen set $K$ contained in their intersection. Let $e$ be the central element of $A$ corresponding to $K$ and note that $a^+$ is the central element corresponding to $M$. Then $a\vee(a^+\wedge e^+)=1$, contrary to $a^+$ being the supplement of $a^+$. Thus the clopen set $M$ is the complement of $N$, showing that $N$ is clopen. So $\alg{A}$ is Hausdorff. 

Suppose conversely that $\alg{A}$ is Hausdorff and its stalks are fsi. For $a\in A$ we have $\lbr a = 1\rbr$ is clopen, so the characteristic function of its complement is continuous, hence a global section, hence an element of $A$. Since the stalks are fsi this section is the supplement of $a$ and is evidently central. 
\end{proof}

\begin{remark}
Priestley duality for (weak) Boolean products, or equivalently globals sections of sheaves over Stone spaces, is well\hyp{}understood, see~\cite{HV-M84,Geh91,GvG18}. If $\alg{A}$ is a (weak) Boolean product of a family $\{\alg{A}_z\}_{z \in Z}$ of bounded distributive lattices indexed by a Stone space $Z$, then there is an order\hyp{}preserving continuous quotient map $q \colon X_{\alg{A}} \twoheadrightarrow Z$ such that for each $z \in Z$, the fiber $q^{-1}(z)$ is homeomorphic and order\hyp{}isomorphic to the dual space of $\alg{A}_z$. Thus combining Propositions~\ref{prop:stalks-of-S(A)}  and~\ref{lk}, we obtain a description of the dual space of $S(\alg{A})$ as a ``sum'' of the subspaces of $X_{\alg{A}}$ of the form ${\uparrow}y$ with $y \in \C(Y)$.  
\end{remark}

\begin{proposition}\label{prop:cs-iff-Hausdorff-rep}
A Heyting algebra is centrally supplemented iff the subdirect embedding $\alg{A}\leq\prod_Y\alg{A}_y$ over the minimum of the dual Esakia space of $\alg{A}$ is a Boolean product representation, and in this case, this Boolean product representation is essentially the usual one. 
\end{proposition}

\begin{proof}
If $\alg{A}$ is centrally supplemented, by Lemma~\ref{min} the minimal prime filters $y\in Y$ are exactly the upsets ${\uparrow}x$ of the ultrafilters $x$ in the center of $\alg{A}$. In this case the congruence $\theta_y$ is equal to the congruence $\theta_x$, and this subdirect representation is equal to the usual Boolean product representation in all ways except the elements of the base space are the upsets ${\uparrow}x$ rather than simple the ultrafilters $x$ themselves. For the converse, we have that the stalks $\alg{A}_y$ are fsi since they are quotients of $\alg{A}$ by prime filters. For $a\in A$, the supplement of $a$ in the product is the characteristic function $\chi_U$ where $U$ is the complement of the set $\lbr a=1\rbr$. Since this is a Boolean product, this equalizer and its complement $U$ are clopen. Thus the patchwork property gives $\chi_U$ belongs to $A$, and is therefore the supplement of $a$ in $\alg{A}$ as well. Clearly this supplement is central. 
\end{proof}

For the usual weak Boolean product representation $\alg{A}\leq\prod_X\alg{A}_x$ we let $\mathcal{D}(\alg{A})$ be the elements of the product that are continuous on a dense open subset of $X$ and let $\theta_{\!D}$ be the relation on the product where $f \mathrel{\theta_{\!D}}g$ iff $f$ and $g$ agree on a dense open subset of $X$. The following can be found in~\cite{CHJ96,Har93}. 

\begin{proposition}
For an algebra $\alg{A}$ with usual weak Boolean product $\alg{A}\leq\prod_X\alg{A}_x$, the collection $\mathcal{D}(\alg{A})$ of dense open sections is a subalgebra of the product and $\theta_{\!D}$ is a congruence on the product. 
\end{proposition}

The quotient algebra $\mathcal{D}(\alg{A})/\theta_{\!D}$ is denoted $\mathfrak{R}\alg{A}$ and called the \emph{algebra of dense open sections}. There is an obvious mapping of $\alg{A}$ into $\mathfrak{R}\alg{A}$, and $\alg{A}$ is called \emph{weakly Hausdorff} if this map is embedding. This happens when two elements $a,b\in A$ that agree on a dense open set are equal. This is the case in any Hausdorff algebra since $\lbr a=b\rbr$ is clopen, so if it contains a dense set it is all of $X$ giving that $a=b$. 

\begin{proposition}
If $\alg{A}$ is a Heyting algebra, then $\alg{A}$ is join\hyp{} and meet\hyp{}densely embedded in the algebra $\mathfrak{R}\,S(\alg{A})$ of dense open sections of $S(\alg{A})$. So the hyper\hyp{}MacNeille completion $\alg{A}^+$ is the MacNeille completion of $\mathfrak{R}\,S(\alg{A})$. 
\end{proposition}

\begin{proof}
Since $S(\alg{A})$ is centrally supplemented, by Proposition~\ref{lk} $S(\alg{A})$ is Hausdorff, and therefore weakly Hausdorff. So the mapping of $S(\alg{A})$ into $\mathfrak{R}\, S(\alg{A})$ is an embedding and then by~\cite[Lem.~6.10]{CHJ96} it is join\hyp{} and meet\hyp{}dense. It follows that the MacNeille completion of $S(\alg{A})$ is equal to the MacNeille completion of $\mathfrak{R}\,S(\alg{A})$. But the MacNeille completion of $S(\alg{A})$ is the hyper\hyp{}MacNeille completion $\alg{A}^+$ by Theorem~\ref{ggg}. 
\end{proof}

As it stands, this is of limited interest. However, results of~\cite{CHJ96,Har93} relate the MacNeille completion of an algebra to its algebra of dense open sections under certain conditions. In particular, we have the following consequence of~\cite[Prop.~4]{Har93}. 

\begin{proposition} \label{h}
If $\alg{A}$ is a weakly Hausdorff Heyting algebra and the stalks of its weak Boolean product representation have uniformly bounded finite cardinality on a dense open subset of the base space, then $\mathfrak{R}\alg{A}$ is the MacNeille completion of $\alg{A}$. 
\end{proposition}

\begin{corollary}
If $\alg{A}$ is a centrally supplemented Heyting algebra and the stalks of its weak Boolean product representation have uniformly bounded finite cardinality, then its hyper\hyp{}MacNeille completion is its algebra $\mathfrak{R}\alg{A}$ of dense open sections. 
\end{corollary}

\begin{proof}
Since $\alg{A}$ is centrally supplemented it is Hausdorff by Proposition~\ref{lk} and its hyper\hyp{}MacNeille completion is its MacNeille completion by Theorem~\ref{j}. The result then follows by Proposition~\ref{h}. 
\end{proof}

We note that $\mathfrak{R}\alg{A}$ always belongs to to the variety generated by $\alg{A}$ since it is a quotient of a subalgebra of a product of quotients of $\alg{A}$. This will be put to use in the following section, but first we turn to a different application of sheaves. 

\begin{definition}
For the subdirect product $\alg{A}\leq\prod_Y\alg{A}_y$ of a Heyting algebra $\alg{A}$ over the minimum of its dual Esakia space, let $S_Y$ be the disjoint union of the stalks $\alg{A}_y$ for $y\in Y$, and for $a\in A$ and $U\subseteq Y$ open set $\mathcal{O}(a,U)=\{a(y):y\in U\}$. 
\end{definition}

The following is a simple modification of the standard Pierce sheaf construction whose proof is found in~\cite[Prop.~1]{Har93}. 

\begin{proposition}
For a Heyting algebra $\alg{A}$ the sets of the form $\mathcal{O}(a,U)$ where $a\in A$ and $U\subseteq Y$ is open form a basis of a topology on $S_Y$. With this topology $S_Y$ is a sheaf of Heyting algebras with $\alg{A}$ a Heyting subalgebra of the globals sections $\Gamma(S_Y)$. 
\end{proposition}

We call this the \emph{sheaf over the minimum of} $\alg{A}$. There are differences between this situation and the central sheaf of $\alg{A}$ that comes from the weak Boolean product representation of $\alg{A}$ over its center. The space $Y$ has a basis of clopen sets, but need not be compact, so is generally not a Stone space. So this sheaf cannot come from weak Boolean product. The patchwork property holds only in that one can patch two global sections over a clopen set, but patching two global sections coming from elements of $\alg{A}$ will not usually result in a global section coming from $A$. This and the lack of compactness of the base space $Y$ point to why $\alg{A}$ can be a proper subalgebra of $\Gamma(S_Y)$. 

\begin{proposition}\label{prop:S(A)subG(SY)}
For $\alg{A}$ a Heyting algebra, the global sections $\Gamma(S_Y)$ of the sheaf over its minimum is a centrally supplemented Heyting algebra and the centrally supplemented extension $S(\alg{A})$ is a supplemented subalgebra of $\Gamma(S_Y)$. 
\end{proposition}

\begin{proof}
	By Definition~\ref{S(A)} $S(\alg{A})$ is the supplemented subalgebra of the product $\prod_{y\in Y} \alg{A}_y$ generated by $\alg{A}$. Since we know that $\alg{A}$ is contained in $\Gamma(S_Y)$ and that $\Gamma(S_Y)$ is a Heyting subalgebra of the product $\prod_{y \in Y} \alg{A}_y$, it remains only to show that $\Gamma(S_Y)$ is closed under supplements in the product. Suppose $\sigma$ is a global section of $S_Y$. Since the stalks are fsi, the supplement $\sigma^+$ is the characteristic function $\chi_U$ of $U=\lbr \sigma < 1\rbr$. Let $y\in Y$. If $\chi_U(y)=0$, then $\sigma(y)=1$, so $\sigma(y)\in\mathcal{O}(1,Y)$. By the continuity of $\sigma$ there is a neighborhood $V$ of $y$ where $\sigma$ takes value 1, hence $\chi_U$ takes value 0 on $V$, showing $\chi_U$ is continuous at $y$. Suppose $\chi_U(y)=1$ and that $a\in A$ is such that $\sigma(y)=a(y)$. Then $a(y)<1$. By the continuity of $\sigma$ there is a neighborhood $W$ of $y$ mapped by $\sigma$ into $\mathcal{O}(a,\lbr a<1\rbr)$. It follows that $\sigma$ takes value less than 1 on $W$ and so $\chi_U$ takes value 1 on $W$. Again, $\chi_U$ is continuous at $y$. So $\sigma^+$ is a global section. 
\end{proof}

\begin{definition}
Let $\mathcal{Q}(\alg{A})$ be the algebra of dense open sections of the sheaf $S_Y$ over the minimum of $\alg{A}$, that is, the set of sections of $S_Y$ that are continuous on a dense open set modulo equivalence on a dense open set. 
\end{definition}

As in~\cite{CHJ96} it is straightforward to verify that $\mathcal{Q}(\alg{A})$ is a Heyting algebra, and by construction it belongs to the variety generated by $\alg{A}$. Unlike the situation with the algebra of dense open sections over the central sheaf, there is no special requirement, like weakly Hausdorff, for the natural map from $\alg{A}$ to $\mathcal{Q}(\alg{A})$ to be an embedding. Furthermore, by Proposition~\ref{prop:S(A)subG(SY)} each member of $S(\alg{A})$ is a global section of $S_Y$ and hence continuous everywhere, and therefore there is a natural map from $S(\alg{A})$ to $\mathcal{Q}(\alg{A})$. In fact, we have the following. 

\begin{proposition}
For any Heyting algebra $\alg{A}$, the natural map from $S(\alg{A})$ into $\mathcal{Q}(\alg{A})$ is a join\hyp{}dense and meet\hyp{}dense Heyting algebra embedding.
\end{proposition}

\begin{proof}
Clearly the natural map $\alpha$ from $S(\alg{A})$ to $\mathcal{Q}(\alg{A})$ is a Heyting algebra homomorphism. To see it is an embedding, we need only show that $1$ is the only element of $S(\alg{A})$ mapped by $\alpha$ to 1. Suppose $a\in A$ and $U\subseteq Y$ is clopen, and set $U'=Y\setminus U$. Let $\psi=(a\wedge\chi_U)\vee \chi_{U'}$. If we show that $\alpha(\psi)=1$ implies that $\psi=1$, it follows from the  description of $S(\alg{A})$ given in Proposition~\ref{fr} that $\alpha$ is an embedding. If $\psi\neq 1$, then there is $y\in U$ with $a(y)\neq 1$. Then $\lbr a<1\rbr\cap U$ is a non\hyp{}empty clopen set on which $\psi$ is different than 1, hence $\alpha(\psi)\neq 1$. The proof of join\hyp{} and meet\hyp{}density is the same as that given in~\cite[Lem.~6.9]{CHJ96} once it is noted that the elements agreeing with a given $a\in A$ on a clopen set $U$ and 0 elsewhere belong to $S(\alg{A})$ and the elements agreeing with $a$ on $U$ and $1$ elsewhere belong to $S(\alg{A})$. 
\end{proof}

This immediately gives the following. 

\begin{theorem} \label{kj}
For a Heyting algebra $\alg{A}$ we have that the hyper\hyp{}MacNeille completion of $\alg{A}$ is the MacNeille completion of the algebra of dense open sections of the sheaf over the minimum of \alg{A}, that is, $\alg{A}^+=\overline{\mathcal{Q}(\alg{A})}$.
\end{theorem}

A number of special cases of this are worthwhile to note. In the next result we will use $Y_0$ for the set of isolated points of the minimum of the dual Esakia space of a Heyting algebra $\alg{A}$. This is clearly an open set. 

\begin{corollary}\label{prop:HMN-iso-dense}
If $\alg{A}$ is a Heyting algebra and the isolated points $Y_0$ of its minimum is a dense subset of $Y$, then $\alg{A}^+ = \prod_{Y_0} \overline{\alg{A}_y}$.
\end{corollary} 

\begin{proof}
If $Y_0$ is dense, then it is a dense open subset of $Y$. So $\mathcal{Q}(\alg{A})$ is isomorphic to the product $\prod_{Y_0}\alg{A}_y$. The result follows from Theorem~\ref{kj} and the well\hyp{}known, and easily proved fact that the MacNeille completion of a product of bounded lattices is the product of their MacNeille completions. 
\end{proof}

\begin{corollary}
If $\alg{A}$ is a Heyting algebra whose minimum is finite, then $\alg{A}^+ = \prod_Y \overline{\alg{A}_y}$. 
\end{corollary}

Theorem~\ref{kj} also has application in another direction.  

\begin{proposition} \label{dr}
If there is a dense open set on which the stalks of the sheaf $S_Y$ over the minimum of $\alg{A}$ have uniformly bounded finite cardinality, then $\alg{A}^+=\mathcal{Q}(\alg{A})$. 
\end{proposition}

\begin{proof}
It is enough to show that in this circumstance $\mathcal{Q}(\alg{A})$ is complete. This is shown by adapting the proof of~\cite[Prop.~4]{Har93}. The key element is that in our setting, for $a,b\in A$ we have that $\lbr a=b\rbr$ is equal to $\lbr a\leftrightarrow b = 1\rbr$ and is therefore clopen. 
\end{proof}

\begin{remark}
Using sheaf representations to construct different types of completions of lattice based algebras is by no means a new technique as references to~\cite{Har93, CHJ96} show. We point to two more examples of this phenomenon. Given a completely regular Baire space $X$ the Dedekind completion of the Riesz space $C(X)$ of real\hyp{}valued continuous functions on $X$ may be obtain as the Riesz space consisting of certain bounded real\hyp{}valued functions on $X$ which are continuous on a dense set and identified by equality on a dense set~\cite[Thm.~6.1]{Dan16}. Similarly, lateral completions of $\ell$\hyp{}groups may be obtain from sheaf representations in a way resembling the construction of $\mathcal{Q}(\alg{A})$~\cite{RY09}.
\end{remark}


\section{Varieties closed under hyper\hyp{}MacNeille completions} \label{examples}

As described in the introduction, the hyper\hyp{}MacNeille completion arose from systematic considerations in substructural proof theory about the admissibility of the cut\hyp{}rule in certain hypersequent calculi. In the context of this work there is a hierarchy of families of equations/formulas, first given in~\cite{CGT08}, among which the formulas belonging to the level $\mathcal{P}_3$ can effectively be transformed into an equivalent hypersequent calculus in which, under certain assumptions, the cut\hyp{}rule is guaranteed to be admissible~\cite{CGT17}. This fact is closely related to one of the key results about hyper\hyp{}MacNeille completions of Heyting algebras. 

\begin{theorem}[{\cite[Thm.~7.3]{CGT17}}] \label{P3}
If $\vty{V}$ is a variety of Heyting algebras axiomatized by $\mathcal{P}_3$\hyp{}equations, then $\vty{V}$ is closed under hyper\hyp{}MacNeille completions. 
\end{theorem}

To give a bit of a feel for such varieties, we give the following. 

\begin{proposition}\label{prop:P3-properties}
If $\vty{V}$ is a variety of Heyting algebras axiomatized by $\mathcal{P}_3$\hyp{}equations, then
\begin{enumerate}
\item The variety $\vty{V}$ is generated by its finite members.
\item If $\alg{A}\in\vty{V}$ is fsi, each Heyting algebra $\alg{B}$ that is a $(\wedge, 0, 1)$\hyp{}subalgebra of $\alg{A}$ belongs to $\vty{V}$.  
\item The variety $\vty{V}$ can be axiomatized using only $\lor$\hyp{}free equations. 
\item The class of fsi algebras in $\vty{V}$ is closed under MacNeille completions.
\item If $\vty{V}\neq \mathbb{HA}$, then there is a natural number $n$ such that the dual Esakia space of any fsi algebra in $\vty{V}$ has at most $n$ maximal points. 
\end{enumerate}
\end{proposition}

\begin{proof}
The first three statements can be found in~\cite{Lau19a} and~\cite[Chap.~2]{Lau19b}, the fourth in~\cite{CGT11}, and the last in~\cite[Chap.~3]{Lau19b}.  
\end{proof}

It is worthwhile to discuss the general situation for completions of Heyting algebras. The variety $\mathbb{HA}$ of Heyting algebras is closed under MacNeille completions, canonical completions, and hyper\hyp{}MacNeille completions. MacNeille completions are always regular, hyper\hyp{}MacNeille completions are regular exactly when the algebra is externally distributive, and canonical completions are regular only for finite Heyting algebras. Here are some known results about completions of Heyting algebras. 

\begin{theorem}[\cite{HB04}] 
The only varieties of Heyting algebras that are closed under MacNeille completions are the trivial variety, the variety of Boolean algebras, and the variety of all Heyting algebras. 
\end{theorem}

\begin{theorem}[\cite{Har08}]
The variety $\vty{V}(\mathbf{3})$ generated by the 3\hyp{}element Heyting algebra admits a regular completion.  
\end{theorem}

To help place the hyper\hyp{}MacNeille completion among these other types of completions, we give the following. 

\begin{theorem}
Let $\vty{V}$ be a variety of Heyting algebras. If $\vty{V}$ is closed under MacNeille completions, then it is closed under hyper\hyp{}MacNeille completions, and if $\vty{V}$ is closed under hyper\hyp{}MacNeille completions, then it is closed under canonical completions. 
\end{theorem}

\begin{proof}
The first statement follows since $\alg{A}\in\vty{V}$ implies that $S(\alg{A})\in\vty{V}$, and $\alg{A}^+=\overline{S(\alg{A})}$. For the second statement suppose that $\vty{V}$ is closed under hyper\hyp{}MacNeille completions and $\alg{A}\in\vty{V}$. By~\cite{GHV06}, for any Heyting algebra $\alg{E}$ the canonical completion $\alg{E}^\sigma$ is a subalgebra of the MacNeille completion of an ultrapower $\alg{E}^I/\mathfrak{u}$ of $\alg{E}$. So for $\alg{E}=S(\alg{A})$ we have that $\alg{E}\in\vty{V}$ and that $\alg{E}$ is centrally supplemented. So the ultrapower $\alg{E}^I/\mathfrak{u}$ also belongs to $\vty{V}$ and is centrally supplemented. Thus by Theorem~\ref{j}(3), $(\alg{E}^I/\mathfrak{u})^+=\overline{\alg{E}^I/\mathfrak{u}}$ belongs to $\vty{V}$. But $\alg{E}^\sigma$ is a subalgebra of $\overline{\alg{E}^I/\mathfrak{u}}$, and so belongs to $\vty{V}$. Since $\alg{A}\leq S(\alg{A})=\alg{E}$, general properties of canonical extensions, see, e.g.,~\cite{GH01}, give $\alg{A}^\sigma\leq \alg{E}^\sigma$, and therefore $\alg{A}^\sigma\in\vty{V}$. 
\end{proof}

One might hope that hyper\hyp{}MacNeille completions will provide a new source of regular completions for Heyting algebras. This in not the case. 

\begin{theorem}
The trivial variety and variety of Boolean algebras are the only varieties of Heyting algebras that are closed under hyper\hyp{}MacNeille completions and for which the hyper\hyp{}MacNeille completion is always regular. 
\end{theorem}

\begin{proof}
Since there are members of $\vty{V}(\mathbf{3})$, such as the algebra of empty and co\hyp{}finite subsets of the natural numbers, that are not externally distributive, the statement follows from Theorem~\ref{j}(7) and the fact that any variety of Heyting algebras other than the trivial variety and the variety of Boolean algebras contains $\vty{V}(\mathbf{3})$. 
\end{proof}

To continue to place the hyper\hyp{}MacNeille completion, we give further examples of varieties that are closed under it.

\begin{theorem}
Any variety of Heyting algebras generated by a finite algebra is closed under hyper\hyp{}MacNeille completions. 
\end{theorem}

\begin{proof}
Suppose $\vty{V}$ is generated by the finite algebra $\alg{E}$ of cardinality $n$. Then by J\'{o}nsson's Lemma~\cite[Sec.~6]{Jon67}, the fsi algebras in $\vty{V}$ are all homomorphic images of subalgebras of $\alg{E}$, hence have cardinality at most $n$. Then by Proposition~\ref{dr} for any $\alg{A}\in\vty{V}$ we have $\alg{A}^+=\mathcal{Q}(\alg{A})$ since all the stalks of the sheaf $S_Y$ over the minimum of $\alg{A}$ are fsi and hence have cardinality at most $n$. But $\mathcal{Q}(\alg{A})$ belongs to the variety generated by $\alg{A}$, and hence to $\vty{V}$. 
\end{proof}

\begin{remark}
Using Proposition~\ref{prop:P3-properties}(2) it is not hard to come up with examples of finitely generated varieties of Heyting algebras which are not axiomatizable by $\mathcal{P}_3$\hyp{}equations, e.g., the variety generated by the Heyting algebra $(\mathbf{2} \times \mathbf{2}) \oplus \mathbf{1}$, obtained by adding a new top element to  the four element Boolean algebra. Note however, that not all varieties determined by $\mathcal{P}_3$\hyp{}equations are finitely generated, such as the variety $\vty{LC}$ generated by all totally ordered Heyting algebras. 
\end{remark}

There are other examples of varieties that are closed under hyper\hyp{}MacNeille completions that are not axiomatizable by $\mathcal{P}_3$\hyp{}equations. The key is the following simple consequence of the fact that $A^+=\overline{S(A)}$ and that $S(A)$ is centrally supplemented. 

\begin{corollary}\label{cor:HMN-iff-MN-on-DM}
A variety $\vty{V}$ of Heyting algebras is closed under hyper\hyp{}MacNeille completions iff the class of centrally supplemented algebras in $\vty{V}$ is closed under MacNeille completions. 
\end{corollary}

Having a supplement makes it possible to use syntactic methods analogous to the ones developed in~\cite{GV99,TV07}, or, alternatively, ALBA-style methods as in, e.g.,~\cite{CP12,CGP14,CPS17}, to establish closure under MacNeille completions. Consider, for example, the variety $\vty{BD}_2$ of Heyting algebras satisfying the equation 
\[
1 \approx x_2 \lor (x_2 \to (x_1 \lor x_1^*)). \tag{$\mathsf{bd}_2$}
\] 
This variety is not determined by $\mathcal{P}_3$\hyp{}equations~\cite[Prop.~3.24]{Lau19a}, nor is it finitely generated. Note however, that the corresponding intermediate logic does admit a simple cut\hyp{}free hypersequent calculus~\cite{CMS13}. 

\begin{proposition}
The variety $\vty{BD}_2$ is closed under hyper\hyp{}MacNeille completions. 
\end{proposition}

\begin{proof}
By Corollary~\ref{cor:HMN-iff-MN-on-DM} it is enough to show that the MacNeille completion of each \mbox{centrally} supplemented member of $\vty{BD}_2$ is in $\vty{BD}_2$. Observe that on supplemented Heyting algebras the defining equation for $\vty{BD}_2$ is equivalent to the equation 
\begin{equation}\label{eq:BD2}
x_1^+ \land x_1 \leq x_2 \lor x_2^*.  
\end{equation}
Let $\alg{A}$ be a supplemented Heyting algebra, then $\overline{\alg{A}}$ is also a supplemented Heyting algebra. Since the set $A$ is both join\hyp{} and meet\hyp{}dense in $\overline{A}$  we have that $\overline{\alg{A}}$ validates Equation~(\ref{eq:BD2}) iff the following two\hyp{}sorted quasi\hyp{}equation holds 
\begin{equation}\label{eq:BD2a}
\forall a, b \in A\, \forall x_1, x_2 \in \overline{A} \ ((a \leq x_1^+ \land x_1 \; \text{and} \; x_2 \lor x_2^* \leq b) \implies a \leq b.
\end{equation}
This is essentially mimicking the \emph{approximation} step of the ALBA algorithm. Noting that $a\leq x_1$ and $a\leq x_1^+$ implies that $a\leq x_1^+\leq a^+$, and similarly with $x_2, x_2^*\leq b$, we have that (\ref{eq:BD2a}) is equivalent to 
\begin{equation}\label{eq:BD2c}
\forall a, b \in A ((a \leq a^+ \; \text{and} \; b^* \leq b) \implies a \leq b).
\end{equation}   
There is no ambiguity here since supplements and pseudo\hyp{}complements in $\alg{A}$ agree with those in $\overline{A}$ by meet\hyp{}density and join\hyp{}density. This last step is essentially a version of what is known as the \emph{Ackermann Lemma} or the \emph{minimal valuation argument} in the context of correspondence theory, see, e.g.,~\cite{CP12}. Now,~(\ref{eq:BD2c}) only depends on $\alg{A}$, it is expressing that every co\hyp{}dense element is below every dense element. Moreover, the above reasoning also shows that~(\ref{eq:BD2c}) is equivalent to $\alg{A}$ validating~(\ref{eq:BD2}) and hence $\alg{A}$ belonging to $\vty{BD}_2$.
\end{proof}

\begin{remark}
A similar kind of argument works for many other varieties of Heyting algebras. Further, the above result has not used the full power of the centrally supplemented condition, it has only used supplementation. 
\end{remark}


\section{Concluding remarks}\label{sec:concluding-remarks} 

While we have been able to obtain some new results and reprove some already known facts about the hyper\hyp{}MacNeille completion of Heyting algebras many questions remain to be answered. In particular, we would like to have a good enough understanding of $\alg{A}^+$ to be able to derive the fact that any variety of Heyting algebra axiomatized by $\mathcal{P}_3$\hyp{}equations is closed under hyper\hyp{}MacNeille completions~\cite[Thm.~7.3]{CGT17} without recourse to syntactic analysis. We note that for any such variety the class of finitely subdirectly irreducible algebras is always closed under MacNeille completions~\cite{CGT11}, see also~\cite{BHIL18} for an alternative argument. A positive answer to the following question would solve this problem. 

\begin{problem}
Let $\vty{V}$ be a variety of Heyting algebras such that for each fsi $\alg{A}\in\vty{V}$ we have that $\overline{\alg{A}}\in\vty{V}$. Must it be the case that $\vty{V}$ is closed under hyper\hyp{}MacNeille completions?
\end{problem}

We believe sheaf\hyp{}theoretic techniques may be useful in approaching this problem.

\bibliographystyle{abbrv}
\bibliography{bibliography}

\begin{thebibliography}{10}

\bibitem{BD74}
R.~Balbes and {\relax Ph}.~Dwinger.
\newblock {\em Distributive Lattices}.
\newblock University of Missouri Press, Columbia, Mo., 1974.

\bibitem{BJO04}
F.~Belardinelli, P.~Jipsen, and H.~Ono.
\newblock Algebraic aspects of cut elimination.
\newblock {\em Studia Logica}, 77(2):209--240, 2004.

\bibitem{BHIL18}
G.~Bezhanishvili, J.~Harding, J.~Ilin, and F.~M. Lauridsen.
\newblock Mac{N}eille transferability and stable classes of {H}eyting algebras.
\newblock {\em Algebra Universalis}, 79:Art. 55, pp. 21, 2018.

\bibitem{Bez06}
N.~Bezhanishvili.
\newblock {\em Lattices of Intermediate and Cylindric Modal Logics}.
\newblock PhD thesis, University of Amsterdam, 2006.

\bibitem{BdRV01}
P.~Blackburn, M.~de~Rijke, and Y.~Venema.
\newblock {\em Modal Logic}, volume~53 of {\em Cambridge Tracts in Theoretical
  Computer Science}.
\newblock Cambridge University Press, Cambridge, 2001.

\bibitem{BS81}
S.~Burris and H.~P. Sankappanavar.
\newblock {\em A Course in Universal Algebra}, volume~78 of {\em Graduate Texts
  in Mathematics}.
\newblock Springer-Verlag, New York-Berlin, 1981.

\bibitem{BW79}
S.~Burris and H.~Werner.
\newblock Sheaf constructions and their elementary properties.
\newblock {\em Trans. Amer. Math. Soc.}, 248(2):269--309, 1979.

\bibitem{CGT08}
A.~Ciabattoni, N.~Galatos, and K.~Terui.
\newblock From axioms to analytic rules in nonclassical logics.
\newblock In {\em Proceedings of the Twenty-Third Annual {IEEE} Symposium on
  Logic in Computer Science, {LICS} 2008, 24-27 June 2008, Pittsburgh, PA,
  {USA}}, pages 229--240. {IEEE} Computer Society, 2008.

\bibitem{CGT11}
A.~Ciabattoni, N.~Galatos, and K.~Terui.
\newblock Mac{N}eille completions of {FL}-algebras.
\newblock {\em Algebra Universalis}, 66(4):405--420, 2011.

\bibitem{CGT12}
A.~Ciabattoni, N.~Galatos, and K.~Terui.
\newblock Algebraic proof theory for substructural logics: cut-elimination and
  completions.
\newblock {\em Ann. Pure Appl. Logic}, 163(3):266--290, 2012.

\bibitem{CGT17}
A.~Ciabattoni, N.~Galatos, and K.~Terui.
\newblock Algebraic proof theory: hypersequents and hypercompletions.
\newblock {\em Ann. Pure Appl. Logic}, 168(3):693--737, 2017.

\bibitem{CMS13}
A.~Ciabattoni, P.~Maffezioli, and L.~Spendier.
\newblock Hypersequent and labelled calculi for intermediate logics.
\newblock In {\em Automated Reasoning with Analytic Tableaux and Related
  Methods}, volume 8123 of {\em Lecture Notes in Comput. Sci.}, pages 81--96.
  Springer, Heidelberg, 2013.

\bibitem{Com71}
S.~D. Comer.
\newblock Representations by algebras of sections over {Boolean} spaces.
\newblock {\em Pacific J. Math.}, 38:29--38, 1971.

\bibitem{CGP14}
W.~Conradie, S.~Ghilardi, and A.~Palmigiano.
\newblock Unified correspondence.
\newblock In {\em Johan van {Benthem} on Logic and Information Dynamics},
  volume~5 of {\em Outst. Contrib. Log.}, pages 933--975. Springer, Cham, 2014.

\bibitem{CP12}
W.~Conradie and A.~Palmigiano.
\newblock Algorithmic correspondence and canonicity for distributive modal
  logic.
\newblock {\em Ann. Pure Appl. Logic}, 163(3):338--376, 2012.

\bibitem{CPS17}
W.~Conradie, A.~Palmigiano, and S.~Sourabh.
\newblock Algebraic modal correspondence: {Sahlqvist} and beyond.
\newblock {\em J. Log. Algebr. Methods Program.}, 91:60--84, 2017.

\bibitem{CHJ96}
G.~D. Crown, J.~Harding, and M.~F. Janowitz.
\newblock Boolean products of lattices.
\newblock {\em Order}, 13(2):175--205, 1996.

\bibitem{Dav72}
B.~A. Davey.
\newblock {${\mathfrak{M}}$}-{Stone} lattices.
\newblock {\em Canadian J. Math.}, 24:1027--1032, 1972.

\bibitem{Dav73}
B.~A. Davey.
\newblock Sheaf spaces and sheaves of universal algebras.
\newblock {\em Math. Z.}, 134:275--290, 1973.

\bibitem{DP02}
B.~A. Davey and H.~A. Priestley.
\newblock {\em Introduction to Lattices and Order}.
\newblock Cambridge University Press, New York, second edition, 2002.

\bibitem{Dan16}
N.~D\u{a}ne\c{t}.
\newblock The {Dedekind} completion of {$C(X)$} with pointwise discontinuous
  functions.
\newblock In {\em Ordered Structures and Applications}, Trends Math., pages
  111--126. Birkh\"{a}user/Springer, Cham, 2016.

\bibitem{Esa85}
L.~Esakia.
\newblock {\em Heyting Algebras {\rm{I}}. Duality Theory}.
\newblock ``Metsniereba'', Tbilisi, 1985.
\newblock (Russian).

\bibitem{Esa19}
L.~Esakia.
\newblock {\em Heyting Algebras: Duality Theory}, volume~50 of {\em Trends in
  Logic}.
\newblock Springer, 2019.
\newblock Translated from the Russian by A.~Evseev.

\bibitem{GJ13}
N.~Galatos and P.~Jipsen.
\newblock Residuated frames with applications to decidability.
\newblock {\em Trans. Amer. Math. Soc.}, 365(3):1219--1249, 2013.

\bibitem{GO10}
N.~Galatos and H.~Ono.
\newblock Cut elimination and strong separation for substructural logics: an
  algebraic approach.
\newblock {\em Ann. Pure Appl. Logic}, 161(9):1097--1133, 2010.

\bibitem{Geh91}
M.~Gehrke.
\newblock The order structure of {Stone} spaces and the {$T_D$}-separation
  axiom.
\newblock {\em Z. Math. Logik Grundlag. Math.}, 37(1):5--15, 1991.

\bibitem{GH01}
M.~Gehrke and J.~Harding.
\newblock Bounded lattice expansions.
\newblock {\em J. Algebra}, 238(1):345--371, 2001.

\bibitem{GHV06}
M.~Gehrke, J.~Harding, and Y.~Venema.
\newblock {MacNeille} completions and canonical extensions.
\newblock {\em Trans. Amer. Math. Soc.}, 358(2):573--590, 2006.

\bibitem{GvG18}
M.~Gehrke and S.~J. v.~Gool.
\newblock Sheaves and duality.
\newblock {\em J. Pure Appl. Algebra}, 222(8):2164--2180, 2018.

\bibitem{GV99}
S.~Givant and Y.~Venema.
\newblock The preservation of {Sahlqvist} equations in completions of {B}oolean
  algebras with operators.
\newblock {\em Algebra Universalis}, 41(1):47--84, 1999.

\bibitem{GS57}
G.~Gr\"{a}tzer and E.~T. Schmidt.
\newblock On a problem of {M.~H.~Stone}.
\newblock {\em Acta Math. Acad. Sci. Hungar.}, 8:455--460, 1957.

\bibitem{GJLPT18}
G.~Greco, P.~Jipsen, F.~Liang, A.~Palmigiano, and A.~Tzimoulis.
\newblock Algebraic proof theory for {LE}-logics.
\newblock arXiv:1808.04642v1, 2018.

\bibitem{HV-M84}
G.~Hansoul and L.~Vrancken-Mawet.
\newblock D\'{e}compositions bool\'{e}ennes de lattis distributifs born\'{e}s.
\newblock {\em Bull. Soc. Roy. Sci. Li\`ege}, 53(2):88--92, 1984.

\bibitem{Har93}
J.~Harding.
\newblock Completions of orthomodular lattices. {II}.
\newblock {\em Order}, 10(3):283--294, 1993.

\bibitem{Har08}
J.~Harding.
\newblock A regular completion for the variety generated by the three-element
  {Heyting} algebra.
\newblock {\em Houston J. Math.}, 34(3):649--660, 2008.

\bibitem{HB04}
J.~Harding and G.~Bezhanishvili.
\newblock {MacNeille} completions of {Heyting} algebras.
\newblock {\em Houston J. Math.}, 30(4):937--952, 2004.

\bibitem{Jon67}
B.~J\'{o}nsson.
\newblock Algebras whose congruence lattices are distributive.
\newblock {\em Math. Scand.}, 21:110--121 (1968), 1967.

\bibitem{Lau19b}
F.~M. Lauridsen.
\newblock {\em Cuts and Completions: Algebraic aspects of structural proof
  theory}.
\newblock PhD thesis, University of Amsterdam, 2019.

\bibitem{Lau19a}
F.~M. Lauridsen.
\newblock Intermediate logics admitting a structural hypersequent calculus.
\newblock {\em Studia Logica}, 107(2):247--282, 2019.

\bibitem{Mae91}
S.~Maehara.
\newblock Lattice-valued representation of the cut-elimination theorem.
\newblock {\em Tsukuba J. Math.}, 15(2):509--521, 1991.

\bibitem{Oka99}
M.~Okada.
\newblock Phase semantic cut-elimination and normalization proofs of first- and
  higher-order linear logic.
\newblock {\em Theoret. Comput. Sci.}, 227(1-2):333--396, 1999.
\newblock Linear logic, I (Tokyo, 1996).

\bibitem{Oka02}
M.~Okada.
\newblock A uniform semantic proof for cut-elimination and completeness of
  various first and higher order logics.
\newblock {\em Theoret. Comput. Sci.}, 281(1-2):471--498, 2002.

\bibitem{OT99}
M.~Okada and K.~Terui.
\newblock The finite model property for various fragments of intuitionistic
  linear logic.
\newblock {\em J. Symbolic Logic}, 64(2):790--802, 1999.

\bibitem{Ono03}
H.~Ono.
\newblock Closure operators and complete embeddings of residuated lattices.
\newblock {\em Studia Logica}, 74(3):427--440, 2003.

\bibitem{OK85}
H.~Ono and Y.~Komori.
\newblock Logics without the contraction rule.
\newblock {\em J. Symbolic Logic}, 50(1):169--201, 1985.

\bibitem{Pie67}
R.~S. Pierce.
\newblock {\em Modules over Commutative Regular Rings}.
\newblock Memoirs of the American Mathematical Society, No. 70. American
  Mathematical Society, Providence, R.I., 1967.

\bibitem{Pri72}
H.~A. Priestley.
\newblock Ordered topological spaces and the representation of distributive
  lattices.
\newblock {\em Proc. London Math. Soc. (3)}, 24:507--530, 1972.

\bibitem{RY09}
W.~Rump and Y.~C. Yang.
\newblock Lateral completion and structure sheaf of an {A}rchimedean
  {$l$}-group.
\newblock {\em J. Pure Appl. Algebra}, 213(1):136--143, 2009.

\bibitem{San85}
H.~P. Sankappanavar.
\newblock Heyting algebras with dual pseudocomplementation.
\newblock {\em Pacific J. Math.}, 117(2):405--415, 1985.

\bibitem{Spe69B}
T.~P. Speed.
\newblock Some remarks on a class of distributive lattices.
\newblock {\em J. Austral. Math. Soc.}, 9:289--296, 1969.

\bibitem{Spe69A}
T.~P. Speed.
\newblock Two congruences on distributive lattices.
\newblock {\em Bull. Soc. Roy. Sci. Li\`ege}, 38:86--95, 1969.

\bibitem{Ter13}
K.~Terui.
\newblock Herbrand's {Theorem} via {Hypercanonical} {Extensions}.
\newblock Presentation delivered on September 24 at TbiLLC 2013, {G}adauri,
  2013.

\bibitem{Ter18}
K.~Terui.
\newblock {MacNeille} completion and {Buchholz'} omega rule for parameter-free
  second order logics.
\newblock In {\em Computer science logic 2018}, volume 119 of {\em LIPIcs.
  Leibniz Int. Proc. Inform.}, pages Art. No. 37, 19. Schloss Dagstuhl.
  Leibniz-Zent. Inform., Wadern, 2018.

\bibitem{TV07}
M.~Theunissen and Y.~Venema.
\newblock {MacNeille} completions of lattice expansions.
\newblock {\em Algebra Universalis}, 57(2):143--193, 2007.

\bibitem{Vag95}
D.~J. Vaggione.
\newblock Locally {Boolean} spectra.
\newblock {\em Algebra Universalis}, 33(3):319--354, 1995.

\bibitem{W05}
A.~M. Wille.
\newblock A {G}entzen system for involutive residuated lattices.
\newblock {\em Algebra Universalis}, 54(4):449--463, 2005.

\end{thebibliography}

\end{document}